\documentclass[11pt]{article}

\usepackage{verbatim}
\usepackage{color}
\usepackage[colorlinks,linkcolor=blue,
citecolor=purple]{hyperref}
\usepackage{array}
\usepackage{makeidx}
\usepackage[centertags]{amsmath}
\usepackage{amsthm}
\usepackage{newlfont}
\usepackage{amsmath}
\usepackage{amsfonts,amsmath,latexsym}
\usepackage{mathrsfs}
\usepackage{amssymb}
\usepackage{amsbsy}
\usepackage{indentfirst}
\usepackage{mathrsfs}
\usepackage{graphicx}
\usepackage{epstopdf}
\usepackage{subfigure}
\usepackage{caption}
\usepackage{grffile}
\usepackage{float}
\usepackage{amssymb, amsmath}
\usepackage{xcolor}
\usepackage{enumitem}
\usepackage{appendix}

\titlepage
\newlength{\defbaselineskip}
\newtheorem{theorem}{Theorem}[section]
\newtheorem{lemma}{Lemma}[section]
\newtheorem{prop}{Proposition}[section]
\newtheorem{remark}{Remark}[section]

\numberwithin{equation}{section}

\newcommand{\be}{\begin{eqnarray}}
\newcommand{\ee}{\end{eqnarray}}
\newcommand{\bestar}{\begin{eqnarray*}}
\newcommand{\eestar}{\end{eqnarray*}}

\allowdisplaybreaks
\graphicspath{{./graphics/}}

\begin{document}

\title{
Strong Approximations in  the Almost Sure Central Limit Theorem
and  Limit Behavior of the Center of Mass
}
 \author{Zhishui Hu$^{a}$, ~~Wei Wang$^{a}$
 ~~and ~~Liang Dong$^b$
\thanks{Corresponding author. \\
\indent~~E-mail addresses: huzs@ustc.edu.cn (Z. Hu),  weiwang0331@mail.ustc.edu.cn (W. Wang),
 dongliang@njust.edu. cn (L. Dong)}\\
\\
$^{a}$Department of Statistics
of Finance, School of Management\\
 University of Science and Technology of
China\\
 Hefei, Anhui 230026, China \\
 $^{b}$Nanjing University of Science and Technology\\
  Jiangyin, Jiangsu 214400, China
}
\date{}
\maketitle

\begin{abstract}
 In this paper,  we establish an almost sure central limit theorem for a general random sequence under a strong approximation condition. Additionally,  we derive the law of the iterated logarithm for the center of mass  corresponding to a random sequence under a different strong approximation condition.  Applications  to step-reinforced random walks are also discussed.

\vskip 0.2cm
\noindent{\it MSC:}
primary 60F15; %Strong limit theorems
secondary 60F05; %	Central limit and other weak theorems
%60F17, Functional limit theorems; invariance principles
60G50 %	Sums of independent random variables; random walks
\vskip 0.2cm \noindent{\it Keywords:} Strong approximation, almost sure central limit
	theorem,  the center of mass, the law of the
	iterated logarithm, step-reinforced random walk

\end{abstract}

\section{Introduction and main results}\label{sec1}

Let $\{H_n, n\ge 1\}$ be a sequence of real-valued random variables. 
We say that $\{H_n, n\ge 1\}$ satisfies the almost sure central limit theorem (ASCLT) if
there exists an increasing sequence $0\le b_0<b_1<b_2<\cdots$ such that 
$b_n\rightarrow \infty$ and with probability $1$,
\be \label{ASCLT}
\frac1{\log b_n}\sum_{k=1}^n \frac{b_k-b_{k-1}}{b_k}\delta_{H_k} \Rightarrow G,
\ee
where $\delta_{H_k}$ is the Dirac measure at $H_k$, $G$ 
represents the standard Gaussian measure on $(\mathbb{R}, \mathscr{B}(\mathbb{R}))$,
 and ‘‘$\Rightarrow$’’ denotes weak convergence of measures.

 Let $\mathcal{C}_b(\mathbb{R})$ be the class of bounded, continuous  real-valued functions. According to the definition of weak convergence of measures, 
 (\ref{ASCLT}) is equivalent to 
\begin{equation}\label{ASCLTf}
	\frac1{\log b_n}\sum_{k=1}^n \frac{b_k-b_{k-1}}{b_k} f(H_k)\rightarrow \mathbb{E}(f(\mathbf{Z}))~~\mathrm{a.s.}\qquad \mathrm{for~all } ~f\in \mathcal{C}_b(\mathbb{R}),
\end{equation}
where  $\mathbf{Z}$ is a standard normal random variable. Moreover, by letting $\mathcal{C}_L(\mathbb{R})$ denote the class of bounded Lipschitz continuous  functions, the Portmanteau theorem (see, for instance, Theorem 13.16 in \cite{KL2020}) implies that  (\ref{ASCLT})  is also equivalent to
\begin{equation}\label{ASCLTeq}
	\frac1{\log b_n}\sum_{k=1}^n \frac{b_k-b_{k-1}}{b_k} f(H_k)\rightarrow \mathbb{E}(f(\mathbf{Z}))~~\mathrm{a.s.}
\qquad \mathrm{for~all } ~f\in \mathcal{C}_L(\mathbb{R}).
\end{equation}

Consider  a sequence of independent
and identically distributed (i.i.d.) random variables  Let $\{X_n, n\ge 1\}$  with $\mathbb{ E}(X_1)=0$ and $\mathbb{ E}(X_1^2)=1$.  Then, the sequence
$\{\sum_{k=1}^{n}X_k/\sqrt{n}, n\ge 1\}$ satisfies the ASCLT  with $b_n=n$, i.e.,
\begin{equation}\label{ASCLTiid}
	\frac1{\log n}\sum_{k=1}^n \frac{1}{k} f\Big(\frac{\sum_{i=1}^{k}X_i}{\sqrt{k}}\Big)\rightarrow \mathbb{E}(f(\mathbf{Z}))~~\mathrm{a.s.}\qquad \mathrm{for~all } ~f\in \mathcal{C}_b(\mathbb{R}).
\end{equation}

The ASCLT (\ref{ASCLTiid}) was conjectured by L\'{e}vy \cite{L1937} and subsequently proved  by Brosamler \cite{Br1988} and Schatte \cite{S1988} under
slightly stronger moment assumptions, and by Lacey and Phillip \cite{LP1990} under  finite variances.  Berkes \cite{B1995} provided  necessary and sufficient criteria for a sequence of i.i.d. random variables to satisfy  (\ref{ASCLTiid}).
Additionally, the  ASCLT for sums of independent random variables has been studied extensively; see, e.g., \cite{B2001, B1993, ChG2007} and references therein. Beyond  independence,  the ASCLT for martingales and other dependent sequences has also been thoroughly investigated. For instance, see \cite{Be2004, Be2009, Be2008,Ch1996,Ch2000,Ch2001, L2001,L2002, M2001}, among others.

From the equivalent definition (\ref{ASCLTeq}),  it follows  that if
 $\{{H}_n, n\ge 1\}$ satsifies the ASCLT and  $\{ \widetilde{H}_n, n\ge 1\}$ is a random sequence such that
\begin{equation}
	H_n-\widetilde{H}_n\rightarrow 0~~~\mathrm{a.s.},    \label{diffHH}
\end{equation}
then  $\{ \widetilde{H}_n, n\ge 1\}$ also satisfies the ASCLT. Indeed,  from (\ref{diffHH}), we have  $f(H_n)-f(\widetilde{H}_n) \rightarrow 0~a.s.$
for any fixed $f\in \mathcal{C}_L(\mathbb{R})$, which leads to the desired result
by (\ref{ASCLTeq}).

Let  $\{W(t),t>0\}$ be  a standard Brownian motion. Note that for any $f\in \mathcal{C}_b(\mathbb{R})$,
\be
\frac{1}{\log n} \int_A^{n} \frac{1}{t} f\Big(\frac{W(t)}{\sqrt{t}}\Big)dt
\rightarrow \mathbb{E} f(\mathbf{Z})\quad\mathrm{a.s.}    \label{BM}
\ee
holds for any fixed $A>0$ (see (47) in  \cite{B2024}).
Hence, for any increasing sequence $\{b_n\}$ satisfying  $b_n\rightarrow \infty$ and
\bestar
\sup_{b_n\le t\le b_{n+1}} \Big|\frac{W(t)}{\sqrt{t}}-\frac{W(b_n)}{\sqrt{b_n}}\Big|\rightarrow 0~~\mathrm{a.s.},
\eestar
it is clear from (\ref{BM}) that  $\{W(b_n)/\sqrt{b_n}, n\ge 1\}$ satisfies the 
ASCLT.  Consequently, if  $\{M_n,n\ge1\}$ is  a random sequence such that
$
	M_n-W(b_n)=o(\sqrt{b_n})~\mathrm{a.s.}, 
$
then$\{ M_n/\sqrt{b_n},~n\ge 1\}$ satisfies the ASCLT, i.e.,
\be   \label{ASCLTMk}
		\frac{1}{\log b_n}\sum_{k=1}^{n}\frac{b_k-b_{k-1}}{b_k}f\left(\frac{M_k}{\sqrt{b_k}}\right)
		\rightarrow\mathbb{ E}f(\mathbf{Z})\quad\mathrm{a.s.}   \qquad \mathrm{for~all } ~f\in \mathcal{C}_b(\mathbb{R}).
	\ee

In this paper,  we will extend  (\ref{ASCLTMk}) to  encompass more general continous functions.  

\begin{theorem}	\label{mainthm2}
	Let $\{M_n,n\ge1\}$ be  a sequence of  random variables and let $\{W(t),t>0\}$ be  a standard Brownian 
motion. Suppose $\{b_n,n\ge 0\}$ is an  increasing sequence  such that $b_n\rightarrow \infty$, 
	 $b_{n+1}-b_{n}=o(b_n/\log\log b_n)$ and
\begin{equation}\label{Mnbn}
	M_n-W(b_n)=o(\sqrt{b_n})\quad\mathrm{a.s.}, 
\end{equation}	
then 
	\be
		\frac{1}{\log b_n}\sum_{k=1}^{n}\frac{b_k-b_{k-1}}{b_k}f\left(\frac{M_k}{\sqrt{b_k}}\right)
		\rightarrow\mathbb{ E}f(\mathbf{Z})\quad\mathrm{a.s.} \qquad \mathrm{for~all } ~f\in \mathcal{C}_e(\mathbb{R}),
 \label{ASCLTa1}
	\ee
where  $\mathcal{C}_e(\mathbb{R})$ represents the class of continuous functions $f: \mathbb{R}\rightarrow \mathbb{R}$ satisfying $|f(x)|\le e^{\gamma x^2}$ for some $\gamma<1/2$. 
\end{theorem}
\begin{remark} 
Berkes and H\"{o}rmann \cite{B2024} recently proved that (\ref{ASCLTa1}) holds for $M_n=\sum_{k=1}^{n}X_k/\sqrt{b_n}$ under a condition analogous to \eqref{Mnbn}, provided that  $\{X_n, n\ge 1\}$  are independent,  $\mathbb{E}(X_n)=0$, and $b_{n+1}-b_{n}=o(b_n/\log\log b_n)$,   where
$b_n=\sum_{k=1}^n \mathbb{ E}X_k^2$.
	Theorem \ref{mainthm2} extends this result from sums of independent  random variables  to more general sequences.
Berkes and H\"{o}rmann \cite{B2024} also showed that by replacing
(\ref{Mnbn}) with \begin{equation*}
		M_n-W(b_n)=o(b_n^{1/2}\psi(b_n))\quad\mathrm{a.s.}
	\end{equation*}	
for any fixed function $\psi(b_n) \to \infty$,  (\ref{ASCLTa1}) generally becomes  false. Therefore, condition (\ref{Mnbn}) is necessary in some sense.
\end{remark}

Similar methods can  be applied to the
strong limit behavior of the center of mass.
The center of mass 
	 $G_n$ of $M_n$ is defined by
\begin{equation*}\label{Gn11}
	{G}_n=\frac{1}{n}\sum_{k=1}^n M_k.
\end{equation*}
The motion of the centre of mass of a random sequence is of interest from a physical point
of view, particularly when  the sequence models a growing polymer molecule (see, e.g., \cite{A1989, C2011}). 
When $\{M_n, n\ge1\}$ is a simple random 
walk in $\mathbb{R}^d$,
Grill \cite{G1988} investigated the recurrence behavior
 of $G_n$  and confirmed
 Paul Erd\H{o}s's conjecture that $G_n$ will,
with probability $1$, return to a fixed region
around the origin infinitely often if $d=1$,
and only finitely many times if $d\ge 2$.
Lo and Wade \cite{L2019}  extended  Grill's results to   general random walks.

 We can establish the law of the iterated logarithm for  
$\{G_n,n\ge 1\}$ corresponding to a general random sequence $\{M_n, n\ge 1\}$ satisfying a  strong approximation condition 
(\ref{Mnbn1}), which is weaker than \eqref{Mnbn}.

 \begin{theorem}	\label{mainthm13}  
 	Let $\{M_n,n\ge1\}$ be  a sequence of  random variables and let $\{W(t),t>0\}$ be  a standard Brownian 
 	motion. 	 Suppose $\{a_n,n\ge 1\}$ and
$\{b_n,n\ge 1\}$ are  positive sequences  such that  $b_n\uparrow \infty$, $b_{n+1}/b_n\rightarrow 1$  and
 	\begin{equation}\label{Mnbn1}
 		M_n/a_n-W(b_n)=o(\sqrt{b_n\log\log b_n})\quad\mathrm{a.s.}
 	\end{equation}	
If   $a_n$ and
$b_n$ vary regularly with  indices $\rho_1>-1$ 
and  $\rho_2\ge 0$, respectively, 
 then
 	\be
 		\limsup_{n\rightarrow \infty}\frac {1}{\sqrt{2a_n^2b_n\log\log b_n}}\Big|\frac{1}{n}\sum_{k=1}^n M_k\Big|=\sqrt{2}((1+\rho_1+\rho_2)(2+2\rho_1+\rho_2))^{-1/2}\quad \mathrm{a.s.}  \label{barycen}
 	\ee
 \end{theorem}

\begin{remark}
In  Theorem  \ref{mainthm13}, we need 
the condition $b_{n+1}/b_n\rightarrow 1$ to make $\{b_n, n\ge 1\}$ dense enough so that $\limsup_{n\rightarrow \infty} |W(b_n)|/\sqrt{2b_n\log\log b_n}=1$ and the process $\{\eta(b_n, s), 0\le s\le 1\}$ satisfies 
 Strassen's law of iterated logarithm, 
where
$\eta(b_n, s)$ is defined as in (\ref{etats}) below.   Without the
restriction $b_{n+1}/b_n\rightarrow 1$, we can also get from the proof that
\bestar
 		\limsup_{n\rightarrow \infty}\frac {1}{\sqrt{2a_n^2b_n\log\log b_n}}\Big|\frac{1}{n}\sum_{k=1}^n M_k\Big|\le\sqrt{2}((1+\rho_1+\rho_2)(2+2\rho_1+\rho_2))^{-1/2}\quad \mathrm{a.s.}  
 	\eestar
\end{remark}

The rest of this paper is organized as follows. 
In Section \ref{sec2}, we apply Theorems \ref{mainthm2} and \ref{mainthm13} to step-reinforced random walks and present the corresponding 
 ASCLTs and the laws of the iterated logarithm for the  centers of mass. Sections \ref{sec3}, \ref{sec4} and \ref{sec5} provide the proofs of Theorem \ref{mainthm2},  Theorem\ref{mainthm13}, and Theorems \ref{mainthm1}-\ref{limsup}, respectively.  
The proofs of two propositions 
used in Section \ref{sec5} are offered 
in \ref{seca}.

Throughout the paper, we assume that $C$ denotes a generic positive constant that may take a different value in each appearance.

\section{Applications to step-reinforced random walks}\label{sec2}

Let $\{X_n,n\geq1\}$ be a sequence of i.i.d. random variables, and let $\{\varepsilon_n,n\geq2\}$ be a sequence of i.i.d. Bernoulli variables with parameter $p\in [0,1)$.
For each $n\geq2$, let  $U_n$ be a random variable uniformly distributed on $ \{ 1,2, \cdots ,n-1 \} $.   
We assume that all the above random variables $X_1, X_2, \varepsilon_2, U_2, X_3, \varepsilon_3, U_3,\cdots$ are mutually independent. 
Set  $\hat{X}_1=X_1$ and then define recursively for $n\geq2$,
\begin{equation}\label{defhatX}
\hat{X}_n:=\left\{
\begin{aligned}
&X_n,&&\quad \text{if}\ \varepsilon_n=0, \\
&\hat{X}_{U_n}, &&\quad \text{if}\ \varepsilon_n=1.
\end{aligned}
\right.
\end{equation}
The algorithm (\ref{defhatX}) was introduced by Simon \cite{S55} to explain the appearance of
power laws  in a wide range of empirical data.
The sequence of the partial sums
$$\hat{S}_n:=\hat{X}_1+\cdots+\hat{X}_n,\qquad n\geq1$$ 
is termed a positively step-reinforced  random walk. Furthermore, Bertoin \cite{B2023}  recently introduced
the following counterbalancing algorithm. Set $\check{X}_1=X_1$ and recursively define  for $n\ge 2$,
\begin{equation*}
	\label{defcheckX}
	\check{X}_n :=
	\left\{
	\begin{aligned}
		& X_n, &&\quad \text{if}\ \varepsilon_n=0, \\
		& -\check{X}_{U_n},  &&\quad \text{if}\ \varepsilon_n=1.
	\end{aligned}
	\right.
\end{equation*}
The sequence of the partial sums
\begin{equation*}
	\check{S}_n:=\check{X}_1+\cdots+\check{X}_n,\qquad n\geq1
\end{equation*}
is called a negatively step-reinforced  random walk. Note that for $p=0$, both $\hat{S}_n$ and $\check{S}_n$ are    random walks with i.i.d steps.

 Assuming 	$X_1$ follows the Rademacher distribution,  i.e., $\mathbb{P}(X_1=1)=\mathbb{P}(X_1=-1)=1/2$, K\"{u}rsten \cite{K2016} and Bertoin \cite{B2023} pointed
out that $\hat{S}_n$ and $\check{S}_n$ can be regarded as
elephant random walks (ERW) with memory parameters $(1+p)/2$  and $(1-p)/2$, respectively. 
The  ERW,  introduced by Sch\"{u}tz and Trimper \cite{S2004}, is a one-dimensional discrete-time nearest neighbour random walk with a
memory of the entire past.
Specifically, fixing a memory parameter $q \in [0,1)$, suppose that an elephant makes an initial step in $\{-1,1\}$ at time 1. Thereafter, at each time $n\ge 2$, the
elephant remembers one step from the past chosen uniformly at random. With probability $q$,
the elephant repeats this step, and with probability $1-q$, it makes a step in the opposite direction. The ERW has garnered significant interest and has been extensively studied in recent years; see e.g \cite{B2016,Be2017, Be2019, Be2021, C2017}.
If $X_1$ has a symmetric stable distribution,  $\hat{S}_n$ corresponds to the  shark random swim, which has been studied in depth by Businger \cite{Bu2018}.

 In a recent study, Bertoin \cite{B2021}  established  the functional central limit theorem
for $\hat{S}_n$.  Bertenghi  and  Rosales-Ortiz \cite{BR2022} obtained the functional central limit theorem for $\check{S}_n$ 
and the joint functional central limit theorem for $(\sum_{k=1}^n X_k, \hat{S}_n,  \check{S}_n)$. 
Hu and Zhang \cite{HZZY}
established  strong invariance principles for both $\hat{S}_n$ and $\check{S}_n$.

In this paper, we denote $m_i=\mathbb{E}(X_1^i)$ for $i=1,2$ and define 
\begin{equation*}
	{\sigma}^2=m_2-m_1^2,~~~\check{\mu}=\frac{(1-p)m_1}{1+p},~~\check{\sigma}^2=m_2-\check{\mu}^2.
\end{equation*}
We derive the ASCLT for step-reinforced random walks.

\begin{theorem}	\label{mainthm1}
Assuming that $\mathbb{E}|X_1|^{2+\delta}<\infty$  for some $\delta>0$  and  $f\in \mathcal{C}_e(\mathbb{R})$.
	\begin{itemize}
		\item [1.]If  $ p \in [0,1/2)$, then
		\begin{equation}
			\frac1{\log n}\sum_{k=1}^n\frac{1}{k}f\left(\frac{\sqrt{1-2p}}{\sigma\sqrt{k}}(\hat{S}_k-km_1)\right) \rightarrow \mathbb{E}(f(\mathbf{Z})) \quad\mathrm{a.s.}\label{ASC1}
		\end{equation} 
		\item [2.]	If $p=1/2$,  then 
		\begin{equation}
			\frac1{\log\log n}\sum_{k=2}^n\frac1{k\log k}f\left(\frac{\hat{S}_k-km_1}{\sigma\sqrt{k\log k}}\right)\rightarrow \mathbb{E}(f(\mathbf{Z})) \quad\mathrm{a.s.}\label{ASC2}
		\end{equation}
		\item [3.] If  $p\in [0,1)$, then 
		\begin{equation}
			\frac1{\log n}\sum_{k=1}^n\frac1kf\left(\frac{\sqrt{1+2p}}{\check{\sigma}\sqrt{k}}(\check{S}_k-k\check{\mu})\right)\rightarrow \mathbb{E}(f(\mathbf{Z})) \quad\mathrm{a.s.}\label{ASC3}
		\end{equation}
	\end{itemize}
	
\end{theorem}

\begin{remark}
By Theorem \ref{mainthm1}, we can also derive the  ASCLT for the ERW  with parameter $q \in (0, 3/4]$
and  $ f\in \mathcal{C}_e(\mathbb{R})$. This generalizes the result
obatined by Guevara \cite{G2019}, who established the
ASCLTs for $ f\in \mathcal{C}_b(\mathbb{R})$ and for specific functions $f(x)=x^{2r}$, where $r\in \mathbb{N}$.
\end{remark}

If we only consider  the ASCLT for $f\in \mathcal{C}_b(\mathbb{R})$, the moment condition $\mathbb{E}(|X_1|^{2+\delta})<\infty$ for some $\delta>0$  in Theorem 
\ref{mainthm1} can be weakened.

\begin{theorem}\label{mainthm3}
If $\mathbb{E}(X_1^2)<\infty$, then Theorem \ref{mainthm1} remains valid for $f\in \mathcal{C}_b(\mathbb{R})$.
\end{theorem}

We can also consider 
	the centers of mass  of $\hat{S}_n$ and  $\check{S}_n$: 
\bestar
\hat{G}_n=\frac{1}{n}\sum_{k=1}^n \hat{S}_k,~~ ~~~\check{G}_n=\frac{1}{n}\sum_{k=1}^n \check{S}_k .
\eestar
%Bercu and Laulin \cite{Be2021} investigated the asymptotic behavior of the center of mass of the elephant random walk in $\mathbb{R}^d$. Chen and  Laulin \cite{Chen2023} obtained similar results for
%the multidimensional amnesia-reinforced elephant random walk. 
By applying Theorems 1.2 and 1.3 in \cite{HZZY}  along with Theorem \ref{mainthm13}, we can establish the laws of the iterated logarithm for $\hat{G}_n$  and $\check{G}_n$.

\begin{theorem}\label{limsup}
	Suppose that $\mathbb{E}(X_1^2)<\infty$. 
	\begin{itemize}
		\item  [1.]If $ p \in [0,1/2)$,  then
		\be\label{Gn1}
			\limsup_{n\rightarrow \infty} \frac {|\hat{G}_n-m_1n/2|}{\sqrt{2n\log\log n}}  =\frac{\sqrt{2}\sigma}{\sqrt{3(2-p)(1-2p)}}\quad\mathrm{a.s.}
		\ee
		\item  [2.]If $ p =1/2$, then
		\be\label{Gn2}
		\limsup_{n\rightarrow \infty}\frac {|\hat{G}_n-m_1n/2|}{\sqrt{2n\log n\log\log \log n}} =\frac{2\sigma}{3}\quad\mathrm{a.s.}
		\ee
		\item [3.] If $p\in [0,1)$, then 
		\be\label{Gn3}
			\limsup_{n\rightarrow \infty} \frac {|\check{G}_n-\check{\mu}n/2| }{\sqrt{2n\log\log n}}
			=\frac{\sqrt{2}\check{\sigma}}{\sqrt{3(2+p)(1+2p)}}\quad\mathrm{a.s.}
		\ee
	\end{itemize}
\end{theorem}
\begin{remark}  
Bercu and Laulin \cite{Be2021} studied the asymptotic properties of the center of mass of an ERW  in $\mathbb{R}^d$.	 According to  Theorem \ref{limsup}, we can also derive the law of the iterated logarithm for  	the center of mass. Specifically, if $q=3/4$,
	 this result  aligns with (2.12)  in \cite{Be2021} for $d=1$. If $q \in (0, 3/4)$,  considering
$({S}_n)$ as  the ERW  with parameter $q$, Theorem \ref{limsup} implies that
\begin{equation}\label{LERW}
	\limsup_{n\rightarrow \infty}\frac {1}{\sqrt{2n\log \log n}}\Big|\frac{1}{n}\sum_{k=1}^n S_k\Big| =\frac{\sqrt{2}}{\sqrt{3(3-2q)(3-4q)}}\quad\mathrm{a.s.}
\end{equation} 
In contrast,  Theorem 2.2 in Bercu and Laulin \cite{Be2021} provided  the following upper bound
for $d=1$:
\begin{equation*}
	\limsup_{n\rightarrow \infty} \frac {1}{\sqrt{2n\log\log n}} \Big|\frac{1}{n}\sum_{k=1}^n S_k\Big| \le \frac{\sqrt{3}+\sqrt{3-4q}}{\sqrt{12q^2(3-4q)}}\quad\mathrm{a.s.}
\end{equation*}
where the expression on the right-hand side is srictly larger than that on the right-hand side of (\ref{LERW})
for $q\in (0,3/4)$.
\end{remark}

\section{Proof of Theorem \ref{mainthm2} } \label{sec3}
\begin{lemma}\label{leW}
If $b_{n+1}-b_{n}=o(b_n/\log\log b_n)$, then
\bestar
\sup_{b_n\le t\le b_{n+1}} \Big|\frac{W(t)}{\sqrt{t}}-\frac{W(b_n)}{\sqrt{b_n}}\Big|\rightarrow 0\quad\mathrm{a.s.}
\eestar
\end{lemma}

\begin{proof} Observe that
\be\label{I12}
\sup_{b_n\le t\le b_{n+1}} \Big|\frac{W(t)}{\sqrt{t}}-\frac{W(b_n)}{\sqrt{b_n}}\Big|&\le& 
\frac{1}{\sqrt{b_n}}
\sup_{b_n\le t\le b_{n+1}}|W(t)-W(b_n)|+|W(b_n)|\Big(\frac{1}{\sqrt{b_{n}}}-\frac{1}{\sqrt{b_{n+1}}}\Big)\nonumber\\
&:=&I_{1n}+I_{2n}.
\ee
%
%If $b_n-b_{n-1}=o(b_n/\sqrt{\log n})$, then $b_n/b_{n-1}\rightarrow 1$ and
% $b_n<2^n$ (or equivalently $\log \log b_n< \log n+\log\log 2$) for sufficiently large $n$. Hence $b_n-b_{n-1}=o(b_n/\sqrt{\log\log b_n})$.

By the law of iterated logarithm for a standard Brownian motion (see Theorems 1.3.1 and $1.3.1^*$ in \cite{CR1981}), we have
\be
\limsup_{T\rightarrow\infty} \frac{\sup_{0\le t\le T}|W(t)|}{\sqrt{2T\log \log T}}=\limsup_{T\rightarrow\infty} \frac{|W(T)|}{\sqrt{2T\log \log T}}= 1\quad\mathrm{a.s.},  \label{LIL}
\ee
 and hence, by noting that $b_{n+1}/b_{n}\rightarrow 1$ and
 $b_{n+1}-b_{n}=o(b_n/\sqrt{\log\log b_n})$,
\be\label{I2n}
I_{2n} =\frac{|W(b_n)|(b_{n+1}-b_{n})}{\sqrt{b_nb_{n+1}}(\sqrt{b_{n+1}}+\sqrt{b}_{n})}
\le \frac{|W(b_n)|(b_{n+1}-b_{n})}{2b_{n}^{3/2}}\rightarrow 0\quad\mathrm{a.s.}
\ee

For any $\varepsilon>0$, there exists $n_0=n_0(\varepsilon)\in \mathbb{N}$
such that $b_{n+1}-b_n\le \varepsilon b_n/\log\log b_n$ for any $n\ge n_0$.
Hence for $n\ge n_0$,
\bestar
I_{1n}\le \frac{1}{\sqrt{b_n}} \sup_{0\le t\le \varepsilon b_n/\log\log b_n}|W(b_n+t)-W(b_n)|.
\eestar
By applying Theorem 1.2.1 in \cite{CR1981}, we have 
$\limsup_{n\rightarrow \infty} I_{1n}\le \sqrt{2\varepsilon}$~a.s.
Hence $I_{1n}\rightarrow 0$~a.s. by the arbitrariness of $\varepsilon \in (0,1)$.
This together with (\ref{I12}) and (\ref{I2n}) implies Lemma \ref{leW}.
\end{proof}

\begin{proof}
[Proof of Theorem \ref{mainthm2}]
By (47) in  \cite{B2024},  we have that for any $f\in \mathcal{C}_e(\mathbb{R})$,
\be
\frac{1}{\log b_n} \int_A^{b_n} \frac{1}{t} f\Big(\frac{W(t)}{\sqrt{t}}\Big)dt
\rightarrow \mathbb{E} f(\mathbf{Z})\quad\mathrm{a.s.}   \label{ASCLTBM}
\ee
holds for any fixed $A>0$.
Define $d_k=1-b_{k-1}/b_{k}$, then $d_k\rightarrow 0$ as $k\rightarrow \infty$. For any $\varepsilon>0$, there exists $k_0=k_0(\varepsilon)\in \mathbb{N}$
such that $\sup_{k>k_0} d_k<\varepsilon$. 
Therefore, 
\be\label{epl}
&&\frac{1}{\log b_n} \Big| \int_{b_{k_0}}^{b_n} \frac{1}{t} f\Big(\frac{W(t)}{\sqrt{t}}\Big)dt-\sum_{k=k_0+1}^{n}\int_{b_{k-1}}^{b_{k}} \frac{1}{b_k} f\Big(\frac{W(t)}{\sqrt{t}}\Big)dt\Big|\nonumber\\
&\le&
  \frac{\sup_{k>k_0} d_k}{\log b_n}\int_{b_{k_0}}^{b_n} \frac{1}{t} \Big|f\Big(\frac{W(t)}{\sqrt{t}}\Big)\Big|dt
\le  \frac{\varepsilon }{\log b_n}\int_{b_{k_0}}^{b_n} \frac{1}{t} \Big|f\Big(\frac{W(t)}{\sqrt{t}}\Big)\Big|dt
\rightarrow \varepsilon \mathbb{E}|f(\mathbf{Z})|\quad\mathrm{a.s.} \nonumber \qquad
\ee
By the arbitrariness of $\varepsilon>0$,
this together with (\ref{ASCLTBM}) implies that for any $f\in \mathcal{C}_e(\mathbb{R})$,
\be \label{TnfW}
T_n(f, W):=\frac{1}{\log b_n}\sum_{k=2}^{n}\int_{b_{k-1}}^{b_{k}} \frac{1}{b_k} f\Big(\frac{W(t)}{\sqrt{t}}\Big)dt \rightarrow \mathbb{E} f(\mathbf{Z})\quad\mathrm{a.s.} 
\ee

Define
\bestar
T_n(f, M)=\frac{1}{\log b_n}\sum_{k=2}^{n}\frac{b_k-b_{k-1}}{b_k}f\Big(\frac{M_k}{\sqrt{b_k}}\Big).
\eestar
Applying Lemma \ref{leW} yields that
\bestar
\sup_{b_{n-1}\le t\le b_{n}} \Big|\frac{W(t)}{\sqrt{t}}-\frac{W(b_n)}{\sqrt{b_n}}\Big|\rightarrow 0\quad\mathrm{a.s.}
\eestar
Hence by (\ref{Mnbn}), 
\be
\sup_{b_{n-1}\le t\le b_{n}} \Big|\frac{W(t)}{\sqrt{t}}-\frac{M_n}{\sqrt{b_n}}\Big|\rightarrow 0\quad\mathrm{a.s.}   \label{diffWM}
\ee

Let $f$ be a uniformly continuous function.
Then  we have
\bestar
D_n:=\sup_{b_{n-1}\le t\le b_{n}} \Big| f\Big(\frac{M_n}{\sqrt{b_n}}
\Big)-f\Big(\frac{W(t)}{\sqrt{t}}\Big) \Big|\rightarrow 0~~\mathrm{a.s.}
\eestar
By noting that 
\begin{equation*}
	\sum_{k=2}^{n}\frac{b_k-b_{k-1}}{b_k}=\sum_{k=2}^{n} \int_{b_{k-1}}^{b_k}\frac{1}{b_k} dt\sim \int_1^{b_n}
	\frac{1}{t} dt= \log b_n, \label{sumdiffb}
\end{equation*}
and applying Toeplitz's lemma, we have
\bestar
	|T_n(f, M)-T_n(f, W)|	\le\frac{1}{\log b_n}\sum_{k=2}^{n}\frac{b_k-b_{k-1}}{b_k} \cdot D_k
\rightarrow 0~~~~\mathrm{a.s.} 
	\eestar
This, together with (\ref{TnfW}),  implies that (\ref{ASCLTa1})  holds for every uniformly continuous function $f$.

Let $h(x)=e^{\gamma x^2}$ for $x\in \mathbb{R}$. For any $\varepsilon>0$, we  define $h_1^{(\varepsilon)}(x)=h(|x|+\varepsilon)$ and
\bestar
h_2^{(\varepsilon)}(x)=\left\{\begin{array}{ll}
h(|x|-\varepsilon), & |x|\ge \varepsilon;\\
1, & |x|<\varepsilon.
\end{array}
\right.
\eestar
By (\ref{diffWM}),  there exists $\Omega_0$
such that $\mathbb{P}(\Omega_0)=1$   and for any $\epsilon>0$ and any $\omega\in \Omega_0$,   we have
\begin{equation*}
	\sup_{b_{n-1}\le t\le b_{n}}\left|\frac{M_n(\omega)}{\sqrt{b_n}}-\frac{W(t,\omega)}{\sqrt{t}}\right|<\epsilon
\end{equation*}
 for  $n\ge n_0$ with some
$n_0=n_0(\omega, \varepsilon)\in \mathbb{N}$.
According to the definition of $h_1^{(\varepsilon)}(x)$ and $h_2^{(\varepsilon)}(x)$,  we get that for all $\omega\in \Omega_0, ~n \ge  n_0$ and $t\in [b_{n-1}, b_n]$,
\bestar
 h_2^{\varepsilon}\Big(\frac{W(t,\omega)}{\sqrt{t}}\Big)
\le h\Big(\frac{M_n(\omega)}{\sqrt{b_n}}\Big)
\le h_1^{\varepsilon}\Big(\frac{W(t,\omega)}{\sqrt{t}}\Big).
\eestar
Hence by (\ref{TnfW}),
\bestar
\mathbb{E}(h_2^{\varepsilon}(\mathbf{Z}))&=&\liminf_{n\rightarrow \infty} T_n(h_2^{\varepsilon}, W)\le \liminf_{n\rightarrow \infty}T_n(h, M)\\
&\le&\limsup_{n\rightarrow \infty}T_n(h, M)\le \limsup_{n\rightarrow \infty} T_n(h_1^{\varepsilon}, W)
=\mathbb{E}(h_1^{\varepsilon}(\mathbf{Z}))\quad\mathrm{a.s.}
\eestar
By noting that
$\lim_{\varepsilon \rightarrow 0}\mathbb{E}(h_1^{\varepsilon}(\mathbf{Z}))=\lim_{\varepsilon \rightarrow 0}\mathbb{E}(h_2^{\varepsilon}(\mathbf{Z}))=\mathbb{E}(h(\mathbf{Z}))$, 
we have
\be\label{ASCLTh}
T_n(h, M)\rightarrow \mathbb{E}(h(\mathbf{Z}))\quad\mathrm{a.s.}
\ee

As follows, we assume that  $f\in \mathcal{C}_e(\mathbb{R})$.
For any $m\in \mathbb{N}$,
define 
\bestar
g_m(x)=\left\{\begin{array}{ll}
1, & |x|\le m;\\
m+1-|x|, & m<|x|<m+1;\\
0, & |x|>m+1.
\end{array}
\right.
\eestar
Recall that $h(x)=e^{\gamma x^2}$  and define 
\bestar 
h_m(x)=(f(x)-h(x))g_m(x),~\quad~
f_m(x)=h_m(x)+h(x).
\eestar
Note that $h_m(x)$ is uniformly continuous since it is  continuous with a compact support.  Hence
(\ref{ASCLTa1}) holds for $h_m$.
This together with (\ref{ASCLTh}) implies $\lim_{n\rightarrow\infty}T_n(f_m, M)
= \mathbb{E}(f_m(\mathbf{Z})) ~\mathrm{a.s.}$ for any fixed $m\in \mathbb{N}$.
Since $0\le g_m(x)\le 1$ and $f(x)\le h(x)$,  we have 
$
h(x)\ge f_m(x)\ge  f(x).
$
Hence  for any fixed $m\in \mathbb{N}$,
\be
\limsup_{n\rightarrow\infty}T_n(f, M)
\le \limsup_{n\rightarrow\infty}T_n(f_m, M)
= \mathbb{E}(f_m(\mathbf{Z}))\quad\mathrm{a.s.}
\label{ASCLTa2}
\ee

By the monotone convergence theorem, we have
$\mathbb{E}(f_m(\mathbf{Z}))\rightarrow \mathbb{E}(f(\mathbf{Z}))$
since $h(x)=e^{\gamma x^2}\ge f_m(x) \downarrow f(x)$.
Hence, by letting  $m\rightarrow \infty$ in (\ref{ASCLTa2}),
\bestar
\limsup_{n\rightarrow\infty}T_n(f, M)
\le \mathbb{E}(f(\mathbf{Z}))\quad\mathrm{a.s.}
\label{ASCLTa3}
\eestar
By replacing $f(x)$ with $-f(x)$, we can obtain that
\bestar
\liminf_{n\rightarrow\infty}T_n(f, M)
\ge \mathbb{E}(f(\mathbf{Z}))\quad\mathrm{a.s.}
\label{ASCLTa4}
\eestar
Hence (\ref{ASCLTa1}) holds for $f\in \mathcal{C}_e(\mathbb{R})$ and the proof of Theorem \ref{mainthm2} is complete.
\end{proof}

\section{Proof of  Theorem \ref{mainthm13}} \label{sec4}

\begin{lemma} \label{lemma4.1}
 Under the conditions of  Theorem \ref{mainthm13}, we have
\be
\Big|\frac {1}{\sqrt{2a_n^2b_n\log\log b_n}}\cdot\frac{1}{n}\sum_{k=1}^n M_k-I_n \Big|	\rightarrow 0\quad\mathrm{a.s.} ,
\label{diffM}
\ee	
where
\be
	I_n= \int_0^1 \frac{t^{\rho_1}W(b_n t^{\rho_2})}{\sqrt{2b_n\log\log b_n}} dt.  \label{defIn}
\ee
\end{lemma}

\begin{proof}
By convention,  we define $a_0=b_0=0$.
Observe that for any $\varepsilon\in (0,1)$,  we have
\be
&&\int_0^1 \frac{|a_{[tn]}W(b_{[tn]})-a_{n}t^{\rho_1}W(b_{n}t^{\rho_2})|}{\sqrt{a_n^2b_n\log\log b_n}}dt \nonumber\\
&&~~~~\le\int_{\varepsilon}^1 \frac{|a_{[tn]}||W(b_{[tn]})-W(b_{n}t^{\rho_2})|}{\sqrt{a_n^2b_n\log\log b_n}}dt+\int_{\varepsilon}^1 \frac{|a_{[tn]}-a_{n}t^{\rho_1}||W(b_{n}t^{\rho_2})|}{\sqrt{a_n^2b_n\log\log b_n}}dt\nonumber\\
&&~~~~~~~~~~+\int_0^{\varepsilon} \frac{|a_{[tn]}W(b_{[tn]})|+|a_{n}t^{\rho_1}W(b_{n}t^{\rho_2})|}{\sqrt{a_n^2b_n\log\log b_n}}dt \nonumber\\
&&~~~~:=J_{1n}+J_{2n}+J_{3n}.   \label{J123n}
\ee

Applying Theorem 1.5.2 in \cite{BGT1987}  gives that for any fixed  $\varepsilon\in (0,1)$,
\bestar
\sup_{\varepsilon\le t\le 1}\Big|\frac{a_{[tn]}}{a_n}-t^{\rho_1}\Big|\rightarrow 0,\qquad	\sup_{\varepsilon\le t\le 1}\Big|\frac{b_{[tn]}}{b_n}-t^{\rho_2}\Big|\rightarrow 0.
\eestar
By Theorem 1.2.1 in \cite{CR1981},  we have
\bestar
	&& \frac{\sup_{\varepsilon\le t\le 1}|W(b_{[tn]})-W(b_{n}t^{\rho_2})|}{\sqrt{b_n\log\log b_n}}
\rightarrow  0\quad\mathrm{a.s.} 
\eestar
Therefore, by  (\ref{LIL}), we get that for any fixed $\varepsilon\in (0,1)$, 
\be
J_{1n}+J_{2n}&\le& \sup_{\varepsilon\le t\le 1} \frac{|W(b_{[tn]})-W(b_{n}t^{\rho_2})|}{\sqrt{b_n\log\log b_n}}\sup_{\varepsilon\le t\le 1} \frac{a_{[tn]}}{a_n}\nonumber\\
&&~~~~~~~~~~+\sup_{\varepsilon\le t\le 1}\Big|\frac{a_{[tn]}}{a_{n}}-t^{\rho_1}\Big| 
\frac{\sup_{0\le t\le b_n}|W(t)|}{\sqrt{b_n\log\log b_n}} \rightarrow 0~~\mathrm{a.s.}    \label{J12n}
\ee

Recall that  $a_n$  varies regularly with  index  $\rho_1>-1$.
Then  for any $0<\varepsilon<\min\{1, 1+\rho_1\}$,
 by applying Potter's theorem (see Theorem 1.5.6 in \cite{BGT1987}), 
 there exists $n_0\in \mathbb{N}$ such that 
$a_{[tn]}/a_n<2t^{\rho_1-\varepsilon}$ holds for any $n\ge n_0$ and $t\in [n_0/n, 1)$.
Hence for sufficiently large $n$, we have
\be
\int_0^{\varepsilon} \frac{|a_{[tn]}W(b_{[tn]})|}{\sqrt{a_n^2b_n\log\log b_n}}dt
&\le& \frac{\int_{n_0/n}^{\varepsilon} 2t^{\rho_1-\varepsilon} |W(b_{[tn]})|dt}{\sqrt{b_n\log\log b_n}}+\frac{n_0}{n}\frac{\sup_{0\le k\le n_0}|a_{k}W(b_{k})|}{\sqrt{a_n^2b_n\log\log b_n}}\nonumber\\
&\le& \frac{ 2\varepsilon^{1+\rho_1-\varepsilon}}{1+\rho_1-\varepsilon}
\frac{ \sup_{0\le t\le b_n} |W(t)|}{\sqrt{b_n\log\log b_n}}+\frac{n_0}{na_n}\frac{\sup_{0\le k\le n_0}|a_{k}W(b_{k})|}{\sqrt{b_n\log\log b_n}}. \qquad
\label{INT1}
\ee
 Note that $na_n\rightarrow \infty$ since $na_n$ varies regularly with index $1+\rho_1>0$. 
By using   (\ref{LIL}) and (\ref{INT1}),  we have
\bestar
\limsup_{\varepsilon\rightarrow 0}\limsup_{n\rightarrow \infty}\int_0^{\varepsilon} \frac{|a_{[tn]}W(b_{[tn]})|}{\sqrt{a_n^2b_n\log\log b_n}}dt=0~~\mathrm{a.s.}
\eestar
Similarly, we also have
\bestar
\limsup_{\varepsilon\rightarrow 0}\limsup_{n\rightarrow \infty}\frac{\int_0^{\varepsilon} t^{\rho_1}|W(b_{n}t^{\rho_2})|dt}{\sqrt{b_n\log\log b_n}}
\le
\limsup_{\varepsilon\rightarrow 0}\limsup_{n\rightarrow \infty}\frac{\varepsilon^{1+\rho_1}}{1+\rho_1}  \frac{\sup_{0\le t\le b_n}|W(t)|}{\sqrt{b_n\log\log b_n}}=0~~\mathrm{a.s.}
\eestar
Hence
\be
\limsup_{\varepsilon\rightarrow 0}\limsup_{n\rightarrow \infty} J_{3n}=0~~\mathrm{a.s.}  \label{J3n}
\ee

Combining (\ref{J123n}),  (\ref{J12n}) and (\ref{J3n})  yields that
\bestar
\int_0^1 \frac{|a_{[tn]}W(b_{[tn]})-a_{n}t^{\rho_1}W(b_{n}t^{\rho_2})|}{\sqrt{a_n^2b_n\log\log b_n}}dt \rightarrow 0~~\mathrm{a.s.}
\eestar
This implies that
\be
&&\Big|\frac {1}{\sqrt{2a_n^2b_n\log\log b_n}}\cdot\frac{1}{n}\sum_{k=1}^n a_k W(b_k)-I_n \Big|	\nonumber\\
&&~~~~~~\le \int_0^1 \frac{|a_{[tn]}W(b_{[tn]})-a_nt^{\rho_1}W(b_{n}t^{\rho_2})|}{\sqrt{2a_n^2b_n\log\log b_n}} dt+\frac{|W(b_n)|+|W(b_0)|}{n\sqrt{2b_n\log\log b_n}}\rightarrow 0\quad\mathrm{a.s.}, \label{INTa1}
\qquad
\ee
where $I_n$ is defined as in (\ref{defIn}).
By (\ref{Mnbn1}), we can get that
	\begin{equation*}
		\Delta_n:=\frac{1}{\sqrt{b_n\log\log b_n}} \max_{0\le k\le n}|M_{k}/a_k-W(b_{k})|{\rightarrow} 0\quad\mathrm{a.s.},
	\end{equation*}
and hence
	\be
\Big|\frac {1}{\sqrt{2a_n^2b_n\log\log b_n}}\cdot\frac{1}{n}\sum_{k=1}^n (M_k-a_kW(b_{k})) \Big|\le \frac {1}{na_n}\sum_{k=1}^n a_k	\cdot \Delta_n \rightarrow 0\quad\mathrm{a.s.},   \label{INTa2}
	\ee
where we have used the fact that $(\sum_{k=1}^n a_k)/(na_n) \rightarrow 1/(1+\rho_1)$ by Karamata's Tauberian theorem (see, for instance, Theorem 5 in Section XIII.5  of \cite{Feller1971}).
Now (\ref{diffM}) follows by combining (\ref{INTa1}) and (\ref{INTa2}).
\end{proof}

\begin{lemma} \label{lemma4.2}
Define
\be
	\eta(t, s)=\frac{W(ts)}{\sqrt{2t\log\log t}},~~0\le s\le 1, ~t>0.
\label{etats}
	\ee	
If $b_{n+1}/b_n\rightarrow 1$, then we have
\bestar
\sup_{b_n\le t\le b_{n+1}}\sup_{0\le s\le 1} |\eta(t, s)-\eta(b_n, s)|\rightarrow 0~~a.s.
\eestar
\end{lemma}

\begin{proof} 
For any $\varepsilon\in (0,1)$, there exists $n_0=n_0(\varepsilon)\in \mathbb{N}$ such that
$b_{n+1}-b_n\le \varepsilon b_n$ holds for any $n\ge n_0$.
Hence for $n\ge n_0$,
\be
&&\sup_{b_n\le t\le b_{n+1}}\sup_{0\le s\le 1} |\eta(t, s)-\eta(b_n, s)|\nonumber\\
&\le &  \frac{\sup\limits_{b_n\le t\le b_{n+1}}\sup\limits_{0\le s\le 1}|W(ts)-W(b_ns)|}{\sqrt{2b_n\log\log b_n}}+\alpha_n\frac{\sup\limits_{0\le t\le b_{n}}|W(t)|}{\sqrt{2b_{n}\log\log b_{n}}} \nonumber\\
&\le&\frac{\sup\limits_{0\le s\le \varepsilon b_n}\sup\limits_{0\le t\le b_{n}}|W(t+s)-W(t)|}{\sqrt{2b_n\log\log b_n}}+\alpha_n\frac{\sup\limits_{0\le t\le b_{n}}|W(t)|}{\sqrt{2b_{n}\log\log b_{n}}}, \label{diffeta}\qquad
\ee
where 
\bestar
\alpha_n=1-\sqrt{b_n\log\log b_n}/\sqrt{b_{n+1}\log\log b_{n+1}}\rightarrow 0 .
\eestar
 By letting $n\rightarrow \infty$ and then $\varepsilon\rightarrow 0$ in  (\ref{diffeta}),  the desired result follows from (\ref{LIL}) and Corollary 1.2.2 in \cite{CR1981}.
\end{proof}

\begin{proof}[Proof of Theorem \ref{mainthm13}]
If $\rho_2>0$, then
	\begin{eqnarray}
	I_n=\rho_2^{-1}\int_{0}^1 s^{(1+\rho_1-\rho_2)/\rho_2}\eta(b_n, s) ds, \nonumber
	\end{eqnarray}
where $I_n$ and $\eta(b_n, s)$ are defined as in (\ref{defIn}) and (\ref{etats}), respectively.
By Lemma \ref{lemma4.2} and Strassen's law of iterated logarithm for Brownian motion (see Theorem 1.3.2  in  \cite{CR1981}), the process $\{\eta(b_n, s), 0\le s\le 1\}$	
	is relatively compact in $C[0,1]$ with probability one, and  the set of its limit points is 
	\bestar
	\mathscr{S}=\Big\{f: f ~\mbox{is an absolutely continuous function on } [0,1], ~f(0)=0, ~\int_0^1 (f'(x))^2 dx\le 1 \Big\}.
	\eestar
	Thus the set of the limit points of $\{I_n, n\ge 0\}$ is
	\bestar
	\mathscr{I}_I= \Big\{ \rho_2^{-1}\int_{0}^1 x^{(1+\rho_1-\rho_2)/\rho_2}f(x)dx: f\in \mathscr{S} \Big\}.
	\eestar
	Note that for any $ f\in \mathscr{S}$,	
	\bestar
	\rho_2^{-1}\int_{0}^1 x^{(1+\rho_1-\rho_2)/\rho_2}f(x)dx &=& \rho_2^{-1}\int_{0}^1 x^{(1+\rho_1-\rho_2)/\rho_2} \int_0^x f'(t)dt dx\\
	&=&\rho_2^{-1}\int_{0}^1 f'(t)  \int_t^1 x^{(1+\rho_1-\rho_2)/\rho_2} dx dt\\
	&=&(1+\rho_1)^{-1} \int_{0}^1 f'(t) (1-t^{(1+\rho_1)/\rho_2}) dt\\
	&\le&  (1+\rho_1)^{-1} \Big(\int_0^1 (f'(t))^2 dt \int_0^1 (1-t^{(1+\rho_1)/\rho_2}) ^2 dt \Big)^{1/2}\\
	&\le&\sqrt{2}((1+\rho_1+\rho_2)(2+2\rho_1+\rho_2))^{-1/2},   \eestar
	and the above equalities hold when $f'(t)= (1-t^{(1+\rho_1)/\rho_2}) /(\int_0^1   (1-x^{(1+\rho_1)/\rho_2}) ^2 dx)^{1/2} $ and $f(0)=0$.
 We can  get the lower bound similarly,  and hence 
	\bestar
	\mathscr{I}_I=\Big[-\sqrt{2}((1+\rho_1+\rho_2)(2+2\rho_1+\rho_2))^{-1/2}, ~~\sqrt{2}((1+\rho_1+\rho_2)(2+2\rho_1+\rho_2))^{-1/2}\Big].
	\eestar 
	This implies that $	\limsup_{n\rightarrow \infty}|I_n| 
=\sqrt{2}((1+\rho_1+\rho_2)(2+2\rho_1+\rho_2))^{-1/2}$, which proves  (\ref{barycen})  for $\rho_2>0$ by (\ref{diffM}).
	
	If
$\rho_2=0$,  then it follows from (\ref{LIL}) and Lemma \ref{lemma4.2}  that
\bestar
	\limsup_{n\rightarrow \infty}|I_n| =(1+\rho_1)^{-1}\limsup_{n\rightarrow \infty}	 \frac{|W(b_{n})|}{\sqrt{2b_n\log\log b_n}} =(1+\rho_1)^{-1}.
\eestar
Therefore,  (\ref{barycen}) also holds for $\rho_2=0$.
\end{proof}

\section{Proofs of Theorems \ref{mainthm1}-\ref{limsup}} \label{sec5}

Suppose that $\alpha>0$ is a constant.  Define
\be\label{trZn}
Z_n=X_n {I}(|X_n| \le n^{\alpha}),~~~~~~ n\ge 1,
\ee
and let 
$\{\hat{Z}_n,\check{Z}_n,\hat{S}^*_n,\check{S}^*_n, n\ge 1\}$ 
be the corresponding variables by replacing $\{X_n, ~n\ge 1\}$ in the algorithms of 
$\{\hat{X}_n,\check{X}_n, \hat{S}_n, \check{S}_n, n\ge 1\}$ by $\{Z_n, n\ge 1\}$.
Let $\mathscr{F}_0=\{\emptyset, \Omega\},~\mathscr{F}_1=\sigma(X_1)$ and
$$ \mathscr{F}_n = \sigma(\varepsilon_{2},\cdots,\varepsilon_{n},U_2,\cdots,U_n,X_1,\cdots,X_n),~~~~n\ge 2. $$

Let $\hat{\gamma}_n=(n+p)/n$ for $n\geq1$,
$\hat{a}_1=1$  and
\begin{equation*} \label{an}
	\hat{a}_n = \prod_{k=1}^{n-1}\hat{\gamma}_k^{-1} = 
	\frac{\Gamma(n)\Gamma(1+p)}{\Gamma(n+p)},~~~~n\ge 2,
\end{equation*}
where $\Gamma(t)=\int_0^{\infty} x^{t-1}e^{-x}dx,
t>0$, is the Gamma function.
As in \cite{HZZY}, 
 $\{\hat{M}^*_n:=\hat{a}_n (\hat{S}^*_n-\mathbb{E}(\hat{S}^*_n)),~ \mathscr{F}_n, n\ge 1\}$
is a  martingale with martingale differences $\{\hat{Y}_n, n\ge 1\}$,
where
  $\hat{Y}_1:=\hat{M}^*_1= \hat{Z}_{1} - \mathbb{E}( \hat{Z}_{1})$ and for $n\ge 2$,
\begin{eqnarray}\label{hatYn}
\hat{Y}_n:=\hat{M}^*_{n}-\hat{M}^*_{n-1}= \hat{a}_n ( \hat{Z}_{n} - \mathbb{E}( \hat{Z}_{n} | \mathscr{F}_{n-1} )).
\end{eqnarray}

  Similarly,  we  define $\check{\gamma}_n = (n-p)/n$  for $n\ge 1$,~ $\check{a}_1=1$  and
\begin{equation*} \label{checkan}
	\check{a}_n = \prod_{k=1}^{n-1}\check{\gamma}_k^{-1} =\frac{\Gamma(n)\Gamma(1-p)}{\Gamma(n-p)},~~~~n\ge 2.
\end{equation*}
Then  $\{\check{M}^*_n:=\check{a}_n (\check{S}^*_n-\mathbb{E}(\check{S}^*_n)), \mathscr{F}_n, n\ge 1\}$
is a  martingale with martingale differences $\{\check{Y}_n, n\ge 1\}$,
where
  $\check{Y}_1:=\hat{M}^*_1= \check{Z}_{1} - \mathbb{E}( \check{Z}_{1})$, and for $n\ge 2$,
\begin{eqnarray*}
\check{Y}_n:=\check{M}^*_{n}-\check{M}^*_{n-1}= \check{a}_n ( \check{Z}_{n} - \mathbb{E}( \check{Z}_{n} | \mathscr{F}_{n-1} )).
\end{eqnarray*}

\begin{lemma}\label{lemma ansn}
	Let
	$
	\hat{s}_n^2=\sum_{k=1}^n \hat{a}_k^2$ 	 and
	$\check{s}_n^2=\sum_{k=1}^n \check{a}_k^2$. Then we have
	\begin{eqnarray}\label{hatas}
		\frac{\hat{a}_n}{\hat{s}_n}	=\left\{
		\begin{array}{ll}
			\sqrt{1-2p}\ n^{-1/2}\, (1+O(n^{2p-1})), & p \in [0, 1/2),\\
			(n\log n)^{-1/2}\, (1+ O((\log n)^{-1})), & p=1/2,
		\end{array}\right.
	\end{eqnarray}
	and
	\begin{equation}\label{checkas}
		\frac{\check{a}_n}{\check{s}_n}	=
		\sqrt{1+2p}\ n^{-1/2}\, (1+O(n^{-1})),\ \  p\in[0,1).
	\end{equation}
\end{lemma}
\begin{proof} 
By the well-known approximation of the Gamma function:
\be
\Gamma(x+1)=\sqrt{2\pi x}(x/e)^x(1+O(x^{-1})),~~x\rightarrow \infty,
\label{Gammax}
\ee
we have
	\begin{equation}\label{asy of hat a_n}
		\hat{a}_n=\Gamma(n)\Gamma(1+p)/\Gamma(n+p)=\Gamma(1+p)n^{-p}(1+O(n^{-1})).
	\end{equation}
	This implies that
	\be
	\label{asy of hat s_n^2}
	\hat{s}_n^2=\sum_{k=1}^n \hat{a}_k^2\sim\left\{
	\begin{array}{ll}
		(\Gamma(1+p))^2(1-2p)^{-1}n^{1-2p},&\quad p \in [0, 1/2),\\
		(\Gamma(1+p))^2\log n,&\quad p=1/2.
	\end{array}
	\right.
	\ee

	If $p \in [0, 1/2)$,  then it follows from (\ref{asy of hat a_n}) that
	\bestar
	(1-2p)\hat{s}_n^2-n\hat{a}_n^2&=&\sum_{k=1}^{n}(1-2p)\hat{a}_k^2-n\hat{a}_n^2
	=-2p+\sum_{k=2}^{n}\left((1-2p-k)\hat{a}_k^2+(k-1)\hat{a}_{k-1}^2\right)\\
	&=& -2p+\sum_{k=2}^{n}\frac{p^2(k-1)}{(k-1+p)^2}\hat{a}_{k-1}^2=O(1).
	\eestar
	Combining this with (\ref{asy of hat a_n}) and (\ref{asy of hat s_n^2}) shows that 
	\begin{equation*}
		\frac{\hat{a}_n}{\hat{s}_n}-\sqrt{\frac{1-2p}{n}}=\frac{n\hat{a}_n^2-(1-2p)\hat{s}_n^2}{\hat{s}_n\sqrt{n}(\hat{a}_n\sqrt{n}+\sqrt{1-2p}\hat{s}_n)}
		=O(n^{2p-3/2}).
	\end{equation*}
	Hence, \eqref{hatas} holds for $p \in [0, 1/2)$.
	
For $p=1/2$,  we observe that
	\be
	\hat{s}_n^2-n\log n\,\hat{a}_n^2&=&
	1+\sum_{k=2}^{n} \left((1-k\log k)\hat{a}_k^2+(k-1)\log (k-1)\hat{a}_{k-1}^2\right)\nonumber\\
	&=&1+\sum_{k=2}^{n}\left(\frac{(1-k\log k)(k-1)^2}{\left(k-1/2\right)^2}+(k-1)\log (k-1)\right)\hat{a}_{k-1}^2.
	\ee
	By using Taylor's expansion of the function $x\log x$, we have
	$$(k-1)\log (k-1)-k\log k+\log k+1=-1/(2\xi_k) \sim -1/(2k), ~~~k \rightarrow \infty,$$
	where
	$\xi_k \in (k-1, k)$. Therefore, simple calculations show that
	\be  \label{middle hat as}
	&&\frac{(1-k\log k)(k-1)^2}{\left(k-1/2\right)^2}+(k-1)\log (k-1)\nonumber\\
	&=& (k-1)\log (k-1) -k\log k+\log k+1
	-  \Big(k\Big(
	\frac{(k-1)^2}{(k-1/2)^2}-1\Big)+1\Big)\log k+\Big(\frac{(k-1)^2}{(k-1/2)^2}-1\Big) \nonumber\\
	&=& O(\log k /k). 
	\ee
	If  $p=1/2$, 	it follows from  (\ref{asy of hat a_n})$-$(\ref{middle hat as}) that
	$
	\hat{s}_n^2-n\log n\,\hat{a}_n^2=O(1)
	$ and
	\begin{equation*}
		\frac{\hat{a}_n}{\hat{s}_n}-\sqrt{\frac{1}{n\log n}}=\frac{n\log n\,\hat{a}_n^2-\hat{s}_n^2}{\hat{s}_n\sqrt{n\log n}(\hat{s}_n+\hat{a}_n\sqrt{n\log n})}=
		O(n^{-1/2}(\log n)^{-3/2}).
	\end{equation*}	
	Hence (\ref{hatas}) still holds for $p=1/2$. 	

To proves (\ref{checkas}),
applying  (\ref{Gammax}) gives that
	\begin{equation}\label{asy of a_n}
		\check{a}_n=\Gamma(n)\Gamma(1-p)/\Gamma(n-p)=\Gamma(1-p)n^{p}(1+O(n^{-1})).
	\end{equation}
	This implies that
	\be\label{asy of V_n}
	\check{s}_n^2=\sum_{k=1}^n \check{a}_k^2\sim (\Gamma(1-p))^2(1+2p)^{-1}n^{1+2p}.
	\ee
	From (\ref{asy of a_n}) we get that
	\bestar
	(1+2p)\check{s}_n^2-n\check{a}_n^2&=&\sum_{k=1}^{n}(1+2p)\check{a}_k^2-n\check{a}_n^2
	=2p+\sum_{k=2}^{n}\left((1+2p-k)\check{a}_k^2+(k-1)\check{a}_{k-1}^2\right)\\
	&=& 2p+\sum_{k=2}^{n}\frac{p^2(k-1)}{(k-1-p)^2} \check{a}_{k-1}^2=O(n^{2p}).
	\eestar
	Combining this with (\ref{asy of a_n}) and (\ref{asy of V_n}) shows that 
	\begin{equation*}
		\frac{\check{a}_n}{\check{s}_n}-\sqrt{\frac{1+2p}{n}}=\frac{n\check{a}_n^2-(1+2p)\check{s}_n^2}{\check{s}_n\sqrt{n}(\check{a}_n\sqrt{n}+\sqrt{1+2p}\check{s}_n)}
		=O(n^{-3/2}).
	\end{equation*}
	This implies  (\ref{checkas}) and
	the proof of Lemma \ref{lemma ansn}  is complete. 
\end{proof}

\subsection{Proof of Theorems \ref{mainthm1}}

In  order to Theorem \ref{mainthm1}, we need the following strong invariance principles
under the moment condition $\mathbb{E}|X_1|^{2+\delta}<\infty$ for some $\delta \in (0,2)$. The proofs of Propositions \ref{mainthm212} and \ref{mainthm222}
are provided  in \ref{seca}.

	\begin{prop}\label{mainthm212}
	Suppose that $\mathbb{ E}|X_1|^{2+\delta}<\infty$
for some $\delta\in (0,2)$.   Define 
	\begin{eqnarray}
		\hat{\sigma}_n^2= \left\{
		\begin{array}{ll}
		\sigma^2 n/(1-2p), & p \in [0, 1/2),\\
	\sigma^2	n\log n, & p=1/2,
		\end{array}\right.  \label{evarhat2}
	\end{eqnarray}
and $\hat{b}_n=n^{-2p}\hat{\sigma}_n^2$ for $0 \le p\le 1/2$.
	Then  we can redefine $\{\hat{S}_n, 
n\ge 1\}$ on a richer probability space on which there exists a standard Brownian motion $\{W(t), t\ge 0\}$ such that
	\begin{equation}\label{spt}
		\left|\frac{\hat{S}_n-nm_1}{\hat{\sigma}_n}- \frac{W(\hat{b}_n)}{\sqrt{\hat{b}_n}} \right|=o(\hat{\delta}_n)\quad\mathrm{a.s.},
	\end{equation}
where
\be\label{Nas}
\hat{\delta}_n=\left
\{\begin{array}{ll}
	n^{-\delta/(4+2\delta)}\sqrt{\log n}, & p \in [0, 1/(2+\delta)),\\
	n^{p-1/2}\log n, & p \in [1/(2+\delta), 1/2),\\
	\log \log n/\sqrt{\log n}, & p=1/2.
\end{array}
\right.
\ee
\end{prop}

\begin{prop}\label{mainthm222}
	Suppose that $\mathbb{ E}|X_1|^{2+\delta}<\infty$
for some $\delta\in (0,2)$.   Define 
 $\check{\sigma}_n^2=\check{\sigma}^2n/(1+2p)$ and $\check{b}_n= n^{2p }\check{\sigma}_n^2$ for $0 \le p< 1$.
	Then  we can redefine $\{\check{S}_n, 
n\ge 1\}$ on a richer probability space on which there exists a standard Brownian motion $\{W(t), t\ge 0\}$ such that
	\bestar
		\left|\frac{\check{S}_n-n\check{\mu}}{\check{\sigma}_n}- \frac{W(\check{b}_n)}{\sqrt{\check{b}_n}} \right|=o(\check{\delta}_n)\quad\mathrm{a.s.},
	\eestar
where
\bestar
\check{\delta}_n=\left
\{\begin{array}{ll}
	n^{-\delta/(4+2\delta)}\sqrt{\log n}, & p \in [0, 2/(2+\delta)),\\
	n^{(p-1)/2}\log n, & p \in [ 2/(2+\delta), 1).
\end{array}
\right.
\eestar
\end{prop}

\begin{proof}
	[Proof of Theorem \ref{mainthm1}]
We only proves (\ref{ASC1}) since the proofs of \eqref{ASC2} and \eqref{ASC3} are similar and we omit the details. 
Assume that  $p\in [0,1/2)$. 
It follows from Proposition \ref{mainthm212} that we can redefine $\{\hat{S}_n, 
n\ge 1\}$ on a richer probability space on which there exists a standard Brownian motion $\{W(t), t\ge 0\}$ such that
\bestar
|n^{-p}(\hat{S}_n-nm_1)- W(\hat{b}_n)|=o((\hat{b}_n)^{1/2})\quad\mathrm{a.s.},
\eestar
where $\hat{b}_n=\sigma^2 n^{1-2p}/(1-2p)$.
	Applying Theorem \ref{mainthm2}  gives that 
for any $f\in \mathcal{C}_e(\mathbb{R})$,
	\bestar
		\frac{1}{\log \hat{b}_n}\sum_{k=1}^{n}\frac{\hat{b}_k-\hat{b}_{k-1}}{\hat{b}_k}f\left(\frac{\sqrt{1-2p}}{\sigma\sqrt{k}}(\hat{S}_k-km_1)\right)\rightarrow\mathbb{ E}f(\mathbf{Z})\quad\mathrm{a.s.}
	\eestar
By noting that  
$\log \hat{b}_n \sim (1-2p)\log n$ and
\bestar
\frac{\hat{b}_n-\hat{b}_{n-1}}{\hat{b}_n}=1-(1-1/n)^{1-2p}\sim (1-2p)/n,
\eestar
and using similar arguments as in the proof of
(\ref{TnfW})  gives that for any $f\in \mathcal{C}_e(\mathbb{R})$,
\begin{equation*}
		\frac{1}{\log n}\sum_{k=1}^{n}\frac{1}{k}f\left(\frac{\sqrt{1-2p}}{\sigma\sqrt{k}}(\hat{S}_k-km_1)\right)\rightarrow \mathbb{ E}f(\mathbf{Z})\quad\mathrm{a.s.}
\end{equation*}
This proves (\ref{ASC1}). 
\end{proof}

\subsection{Proof of Theorem \ref{mainthm3}}

The proof of Theorem \ref{mainthm3} is based on a version of ASCLT for martingales in  \cite{Ch1996},
which is stated as follows.
\begin{lemma}\label{ASCLTfm}
Let $(Y_n,\mathscr{F}_n)_{n\ge 0}$ be a martingale  difference sequence with finite second moments and write $M_n=\sum_{k=1}^n Y_k$.
Suppose that there exists an increasing sequence  $0\le b_0<b_1<\cdots$  such that $b_n\rightarrow \infty$ and the
following conditions hold:
	\begin{itemize}
		\item [(i).] $b_n^{-1}\sum_{k=1}^n\mathbb{E}(Y_k^2|\mathscr{F}_{k-1}) \rightarrow 1~\mathrm{a.s.};$ 
		\item [(ii).] $\sum\nolimits_{n\geq1}b_{n}^{-1}\mathbb{E}(Y_n^2I\{|Y_n|>\varepsilon \sqrt{b_{n}}\}|\mathscr{F}_{n-1})<\infty~\mathrm{a.s.}$ for all $\varepsilon>0$;
		\item [(iii).]  $\sum\nolimits_{n\geq1}b_{n}^{-\beta }\mathbb{E}(|Y_n|^{2\beta}I\{|Y_n|\leq \sqrt{b_{n}}\}|\mathscr{F}_{n-1})<\infty~\mathrm{a.s.}$ for some $\beta>0$.
	\end{itemize}
Then  for any  $f \in \mathcal{C}_b( \mathbb{R})$,
	\begin{equation*}
		\frac1{\log b_n}\sum_{k=1}^n\frac{b_k-b_{k-1}}{b_k}f\Big(\frac{M_k}{\sqrt{b_k}}\Big)\rightarrow \mathbb{E}(f(\mathbf{Z}))\quad\mathrm{a.s.}\label{origin}
	\end{equation*}
\end{lemma}

\vskip 0.2cm

\begin{proof}[Proof of Theorem \ref{mainthm3}]
We only prove (\ref{ASC1})  for $f \in \mathcal{C}_b( \mathbb{R})$ since the  proofs of the others  are similar and we omit the details. By the equivalence of  (\ref{ASCLTf}) and (\ref{ASCLTeq}), it is sufficient to show that for any $p \in [0,1/2)$ and any $f\in \mathcal{C}_L( \mathbb{R})$,
\begin{equation}\label{ACS}
	\frac1{\log n}\sum_{k=1}^n \frac{1}{k} f\left(\frac{\sqrt{1-2p}}{\sigma\sqrt{k}}(\hat{S}_k-km_1)\right)\rightarrow \mathbb{E}(f(\mathbf{Z}))\quad\mathrm{a.s.}
\end{equation}
Take $\alpha=1/2$ in (\ref{trZn}).
Note that $\mbox{Var}(\hat{S}^*_n)\sim \sigma^2n/(1-2p)$ for $p<1/2$ by Lemma 4.6 in \cite{HZZY}.
Applying (\ref{asy of hat a_n}) and (\ref{asy of hat s_n^2}) yields that
\begin{equation*}
	\frac{\text{Var}(\hat{M}_n^*)}{\sigma^2\hat{s}_n^2}=\frac{\hat{a}_n^2\text{Var}(\hat{S}^*_n)}{\sigma^2\hat{s}_n^2} \rightarrow 1.
\end{equation*}
	From 	Lemma \ref{lemmaa4} below, we have 
	$
\mathbb{E}(\hat{Y}_n^2|\mathscr{F}_{n-1})
-\mathbb{E}(\hat{Y}_n^2)=o(1)\hat{a}_n^2, 
$ and hence
\be
		\frac{1}{\sigma^2\hat{s}_n^2}\sum_{k=1}^n\mathbb{E}(\hat{Y}_k^2|\mathscr{F}_{k-1})=\frac{\text{Var}(\hat{M}_n^*)}{\sigma^2\hat{s}_n^2}+\frac{\sum_{k=1}^n(\mathbb{E}(\hat{Y}_k^2|\mathscr{F}_{k-1})
-\mathbb{E}(\hat{Y}_k^2))}{\sigma^2\sum_{k=1}^n\hat{a}_k^2}\rightarrow1\quad\mathrm{a.s.}\label{con1}
	\ee
Similar arguments  as in the proof Lemma 4.1 in \cite{HZZY} show
that for any $m, n\in \mathbb{N}$,
\bestar
\mathbb{E}(\hat{Z}_n^{2m})=\mathbb{E}(\check{Z}_n^{2m})\le\mathbb{E}({Z}_n^{2m}).
\label{Zmoment}
\eestar
According to (\ref{hatYn}),  we have
\be\label{yn4}
\mathbb{E}(\hat{Y}_{n}^4)\le 8\hat{a}_n^4 \Big(\mathbb{E}(\hat{Z}_{n}^4)+\mathbb{E}\big( \mathbb{E}^4\big( \hat{Z}_{n} \big| \mathscr{F}_{n-1} \big)\big)\Big)\le 16\hat{a}_n^4 \mathbb{E}(\hat{Z}_{n}^4)\le 16\hat{a}_n^4 \mathbb{E}({Z}_{n}^4). 
\ee
Applying (\ref{hatas})  yields 
\begin{eqnarray*}
\sum_{n=1}^\infty\hat{s}_{n}^{-4}\mathbb{E}(\hat{Y}_n^4)&\le&	C\sum_{n=1}^{\infty}\frac{\mathbb{E}(Z_n^4)}{n^2}	= C\mathbb{E}\Big(\sum_{n=1}^{\infty}\frac{X_1^4I(|X_1| \le\sqrt{n})}{n^2}\Big)\nonumber\\
	&=&C\mathbb{E}\Big(X_1^4\sum_{n\ge X_1^2}n^{-2}\Big)
	\le C\mathbb{E}(X_1^2)<\infty.
\end{eqnarray*}
Hence  for all $\varepsilon>0$,  we have
	\begin{equation}\label{con11}
		\sum_{n=1}^\infty \hat{s}_{n}^{-2}\mathbb{E}(\hat{Y}_n^2I\{|\hat{Y}_n|>\varepsilon \hat{s}_{n}\}|\mathscr{F}_{n-1}) \leq \varepsilon^{-2}\sum_{n=1}^\infty \hat{s}_{n}^{-4}\mathbb{E}(\hat{Y}_n^4|\mathscr{F}_{n-1})<\infty \quad\mathrm{a.s.},
	\end{equation}
	and 
	\begin{equation}\label{con12}
		\sum_{n=1}^\infty \hat{s}_{n}^{-4}\mathbb{E}(\hat{Y}_n^4I\{|\hat{Y}_n|\leq \varepsilon\hat{s}_n\}|\mathscr{F}_{n-1}) \le \sum_{n=1}^\infty \hat{s}_n^{-4}\mathbb{E}(\hat{Y}_n^4|\mathscr{F}_{n-1})<\infty \quad\mathrm{a.s.}
	\end{equation}
	Applying  Lemma \ref{ASCLTfm} with $b_n=\sigma^2\hat{s}_n^2$
and using (\ref{con1}), (\ref{con11}) and (\ref{con12}) gives that for any $f\in \mathcal{C}_b(\mathbb{R})$,
	\begin{equation}
		\frac{1}{\log \hat{s}_n^2}\sum_{k=1}^n\frac{\hat{a}_k^2}{\hat{s}_k^2} f\Big(\frac{\hat{M}_k^*}{\sigma\hat{s}_k}\Big)\rightarrow \mathbb{E}(f(\mathbf{Z}))\quad\mathrm{a.s.} \label{original state}
	\end{equation}
 Corollary 1.1 in  \cite{HZZY} gives that  
\begin{equation*}
	\limsup_{n\rightarrow\infty}\frac{|\hat{S}_n-nm_1|}{\sqrt{2n\log\log n}}=\frac{\sigma}{\sqrt{1-2p}}\quad\mathrm{a.s.}
\end{equation*}	
Combining this with \eqref{hatas}, (\ref{Esn})  and (\ref{diffstars})  yields 
	that
	\begin{eqnarray}
	&&	\Big|\frac{\hat{M}_n^*}{\hat{s}_n}-\frac{\sqrt{1-2p}(\hat{S}_n-nm_1)}{\sqrt{n}}\Big|
	=\Big|\frac{\hat{a}_n(\hat{S}^*_n-\mathbb{E}(\hat{S}^*_n))}{\hat{s}_n}-
		\frac{\sqrt{1-2p}(\hat{S}_n-nm_1)}{\sqrt{n}}\Big|\nonumber\\
		&&~~~~\le \frac{\hat{a}_n\sqrt{n}}{\hat{s}_n}\Big|\frac{\hat{S}^*_n-\hat{S}_n}{\sqrt{n}}-\frac{\mathbb{E}(\hat{S}^*_n)-nm_1}{\sqrt{n}}\Big|+
		\Big|\frac{\hat{a}_n(\hat{S}_n-nm_1)}{\hat{s}_n}-
		\frac{\sqrt{1-2p}(\hat{S}_n-nm_1)}{\sqrt{n}}\Big|\nonumber\\
		&&~~~~=\frac{\hat{a}_n\sqrt{n}}{\hat{s}_n}\Big|\frac{\hat{S}^*_n-\hat{S}_n}{\sqrt{n}}-\frac{\mathbb{E}(\hat{S}^*_n)-nm_1}{\sqrt{n}}\Big|+O(n^{2p-3/2})|\hat{S}_n-nm_1|\rightarrow 0\quad\mathrm{a.s.}\nonumber \label{ms1}
	\end{eqnarray}

As follows, we assume that $f\in \mathcal{C}_L(\mathbb{R})$. Then there exists $L>0$
such that $|f(x)|\le L$ and $|f(x)-f(y)|\le L|x-y|$ for all $x, y\in \mathbb{R}$.
By (\ref{hatas}) and (\ref{asy of hat s_n^2}), we have
\bestar
&& \frac{1}{\log \hat{s}_n^2}\Big|\sum_{k=1}^n \frac{1-2p}{k} f\Big(\frac{\sqrt{1-2p}}{\sigma\sqrt{k}}(\hat{S}_k-km_1)\Big)-\sum_{k=1}^n\frac{\hat{a}_k^2}{\hat{s}_k^2} f\Big(\frac{\hat{M}_k^*}{\sigma\hat{s}_k}\Big)\Big|\\
&\le & \frac{1}{\log \hat{s}_n^2}\sum_{k=1}^n  \frac{1-2p}{k} \Big|f\Big(\frac{\sqrt{1-2p}}{\sigma\sqrt{k}}(\hat{S}_k-km_1)\Big)-f\Big(\frac{\hat{M}_k^*}{\sigma\hat{s}_k}\Big)\Big|+\frac{1}{\log \hat{s}_n^2} \sum_{k=1}^n \Big|\frac{1-2p}{k}- \frac{\hat{a}_k^2}{\hat{s}_k^2}\Big| \Big|f\Big(\frac{\hat{M}_k^*}{\sigma\hat{s}_k}\Big)\Big|\\
&\le&  \frac{L}{\log \hat{s}_n^2} \sum_{k=1}^n \frac{1}{k} \Big|\frac{\sqrt{1-2p}}{\sigma\sqrt{k}}(\hat{S}_k-km_1)-\frac{\hat{M}_k^*}{\sigma\hat{s}_k}\Big|+ \frac{L}{\log \hat{s}_n^2}\sum_{k=1}^n \frac{1}{k} \Big |1- \frac{k\hat{a}_k^2}{(1-2p)\hat{s}_k^2 }\Big| 
\rightarrow 0\quad\mathrm{a.s.}
\eestar
	Combining this with (\ref{original state})   implies 
that for any $f\in \mathcal{C}_L(\mathbb{R})$,
\bestar
\frac{1}{\log \hat{s}_n^2}\sum_{k=1}^n \frac{1-2p}{k} f\Big(\frac{\sqrt{1-2p}}{\sigma\sqrt{k}}(\hat{S}_k-km_1)\Big)\rightarrow \mathbb{E}(f(\mathbf{Z}))\quad\mathrm{a.s.}
\eestar
This, together with (\ref{asy of hat s_n^2}), yields
(\ref{ACS}) 
 and completes the proof of Theorem \ref{mainthm3}.
\end{proof}

\subsection{Proof of Theorem \ref{limsup}} 
\begin{proof}[Proof of Theorem \ref{limsup}]
Let $M_n=\hat{S}_n-m_1n$ and
	\begin{equation*}
	a_n= \left\{
	\begin{array}{ll}
		\frac{\sigma n^{p}}{\sqrt{1-2p}}, & p \in [0,1/2),\\
	\sigma\sqrt{n}, & p=1/2,
	\end{array}\right.  ~~~~~
b_n= \left\{
\begin{array}{ll}
	n^{1-2p}, & p \in [0,1/2),\\
	\log n, & p=1/2.
\end{array}\right.
\end{equation*}
Theorems 1.2  in  \cite{HZZY}  gives that  
\begin{equation*}
	\frac{1}{\sqrt{b_n\log \log b_n}}\big|M_n/a_n-W(b_n)\big|{\rightarrow} 0~~a.s.
\end{equation*}
Hence (\ref{Gn1}) and (\ref{Gn2}) follow immediately by  Theorem \ref{mainthm13}. 	The proof of \eqref{Gn3} is similar and we omit the details. 
\end{proof}

\begin{appendices}
	\renewcommand{\thesection}{Appendix A}
	\renewcommand{\thelemma}{\Alph{section}.\arabic{lemma}}
\renewcommand{\theequation}{\Alph{section}.\arabic{equation}}

\section{ Proof of Propositions \ref{mainthm212} 
	and \ref{mainthm222}}  \label{seca}

In order to prove Propositions \ref{mainthm212} 
	and \ref{mainthm222}, 
 we use the same method as in \cite{HZZY} and start with some preliminaries.

\begin{lemma} \label{lemmaAA5} Suppose that $\mathbb{E}|X_1|^{r}<\infty$,~$r>1$ and $\alpha(r-1)<1$.
	Then  we have \begin{eqnarray}
		\label{Esn}
		\mathbb{E}(\hat{S}^*_n)=m_1n+\left\{
		\begin{array}{ll}
			o(n^{1-\alpha(r-1)}), & p \in  [0, 1-\alpha(r-1)),\\
			O(n^p), &  p \in [1-\alpha(r-1),1),\\
		\end{array}\right. 
	\end{eqnarray}
and for $p\in  [0,1)$,
\begin{eqnarray}
		\label{Echecksn}
		\mathbb{E}(\check{S}^*_n)=\check{\mu}n+
			o(n^{1-\alpha(r-1)}). 
	\end{eqnarray}
\end{lemma}

\begin{proof}  We only prove (\ref{Esn})  since the proof of (\ref{Echecksn}) is similar
and we omit the details.

Note that
\bestar\label{hatZnConE}
\mathbb{E}\big( \hat{Z}_{n+1} \big| \mathscr{F}_n \big) = p \frac{\hat{Z}_1 + \cdots + \hat{Z}_n}{n}+(1-p)\mathbb{E}(Z_{n+1})= \frac{p}{n}\hat{S}^*_n + (1-p)\mathbb{E}(Z_{n+1}).
\eestar
Then
	\bestar
	\mathbb{E}( \hat{S}^*_{n+1})=\frac{n+p}{n}\mathbb{E}(\hat{S}^*_n) + (1-p)\mathbb{E}(Z_{n+1}),\label{indehat}
	\eestar
	and hence
	\bestar
	\tilde{a}_{n+1}=\frac{n+p}{n}\tilde{a}_n +\tilde{b}_{n+1},
	\eestar
where
	$
	\tilde{a}_n=\mathbb{E}(\hat{S}^*_n)-m_1n$ 
and $\tilde{b}_n=-(1-p)\mathbb{E}(X_nI(|X_n|>n^{\alpha})).$
Observe that
	\begin{eqnarray}
		|\tilde{b}_n|\le \mathbb{E}(|X_1|I(|X_1|>n^{\alpha}))\le  n^{-\alpha(r-1)}\mathbb{E}(|X_1|^{r}I(|X_1|>n^{\alpha}))=o(n^{-\alpha(r-1)}).\nonumber
	\end{eqnarray}
	Applying   Lemma 4.2 in \cite{HZZY}   gives that 
	\begin{eqnarray}\label{Esn1}
		\mathbb{E}(\hat{S}^*_n)=m_1n+\left\{
		\begin{array}{ll}
			o(n^{1-\alpha(r-1)}), & p \in  [0, 1-\alpha(r-1)),\\
			O(n^p), &  p \in (1-\alpha(r-1),1).\\
		\end{array}\right.  
	\end{eqnarray}

For $p=1-\alpha(r-1)$,  note that
$
\tilde{a}_{n}=\tilde{a}_1\gamma_{1,n}(p) 
+\sum_{j=2}^n \tilde{b}_{j}\gamma_{j,n}(p),
$
where
\bestar
\gamma_{j,n}(p)=\prod_{k=j}^{n-1} \frac{k+p}{k}=\frac{\Gamma(n+p)}{\Gamma(n)}\frac{\Gamma(j)}{\Gamma(j+p)},~~~~j\le n.
\eestar
Similar arguments as in the proof of Lemma 4.6 in \cite{HZZY} show that
there exists $n_2\in \mathbb{N}$ such that
	\begin{eqnarray*}
		\sum_{j=n_2}^{n} |\tilde{b}_j| \gamma_{j,n}(p) &\le& C n^{p} \sum_{j=n_2}^{n}j^{-p}\mathbb{E}(|X_{1}|I(|X_{1}|>j^{\alpha}))\\
		&\le& C n^{p} \mathbb{E}\Big(|X_{1}|\sum_{j=1}^n j^{-p}I(|X_{1}|>j^{\alpha})\Big)\\
		&\le& C n^{p} \mathbb{E}\Big(|X_{1}|\sum_{|X_1|^{1/\alpha}>j} j^{-p})\Big)\\
		&\le& C n^{p}\mathbb{E}(|X_1|^{1+(1-p)/\alpha})=C n^{p}\mathbb{E}(|X_1|^{r})=O(n^{p}),
	\end{eqnarray*}
and subsequently,  $\mathbb{E}(\hat{S}^*_n)=m_1n+O(n^p)$ 
	holds  for $p=1-\alpha(r-1)$. The proof of  Lemma \ref{lemmaAA5}  is complete.
\end{proof}

\begin{lemma} \label{lemma5} Suppose that $\mathbb{E}|X_1|^{2+\delta}<\infty,~0<\delta<2$ and  $\alpha=1/(2+\delta)$.
	Then 
	\be
	&&\text{Var}(\hat{S}^*_n) =\left\{
	\begin{array}{ll}
		\sigma^2 n/(1-2p)	+o(n^{2/(2+\delta)}), & p \in [0, 1/(2+\delta)),\\
 \sigma^2 n/(1-2p)	+o(n^{2/(2+\delta)}\log n), & p=1/(2+\delta),\\
\sigma^2 n/(1-2p)	+O(n^{2p}), & p \in(1/(2+\delta), 1/2),\\
		\sigma^2 n\log n+O(n),&  p =1/2,
	\end{array}\right. \label{Varsn}
	\ee	
and for $p\in  [0,1) $,
\be
\text{Var}(\check{S}^*_n) ={\check{\sigma}^2n}/({1+2p})+o(n^{2/(2+\delta)}).
\label{Varchecksn}
\ee
\end{lemma}

\begin{proof}   We only prove (\ref{Varsn})  since the proof of (\ref{Varchecksn}) is similar
and we omit the details. 
It follows from the proof of Lemma 
4.6 in \cite{HZZY} that
\be
	\text{Var}\big( \hat{S}^*_{n+1} \big)
=\frac{n+2p}{n}\mbox{Var}(\hat{S}^*_n)+ b_{n+1},  \label{Varrec}
	\ee
where
	\begin{equation}\label{VarS}
		b_{n+1}=\frac{p}{n}\sum_{k=1}^n \mathbb{E}(\hat{Z}_{k}^2) + (1-p)\mathbb{E}(Z_{n+1}^2)-\Big( \frac{p}{n}\mathbb{E}(\hat{S}^*_n) + (1-p)\mathbb{E}(Z_{n+1}) \Big)^2.
	\end{equation}

Note that
	\be
		|\mathbb{E}(Z_{n+1}^i)-m_i|&=&|\mathbb{E}(X^i_{1}I(|X_{1}|>n^{\alpha}))|\nonumber\\
&\le& n^{-\alpha(2-i+\delta)}\mathbb{E}(|X_{1}|^{2+\delta}I(|X_{1}|>n^{\alpha}))=o(n^{-\alpha(2-i+\delta)}),~~~i=1,2.~~~~   \label{remZi}
	\ee
By  using (\ref{Esn1}) and recalling that $\alpha=1/(2+\delta)$, we have
	\be 
	\Big( \frac{p}{n}\mathbb{E}(\hat{S}^*_n) + (1-p)\mathbb{E}(Z_{n+1}) \Big)^2-m_1^2
&=&O(1) \Big( \frac{p}{n}\mathbb{E}(\hat{S}^*_n) + (1-p)\mathbb{E}(Z_{n+1}) -m_1\Big)\nonumber\\
&=&\left\{
		\begin{array}{ll}
			o(n^{-(1+\delta)/(2+\delta)}), & p \in [0, 1/(2+\delta)),\\
			O(n^{-(1-p)}), &  p \in [1/(2+\delta),1).\\
		\end{array}\right.    \label{sasasa}
	\ee
By replacing $Z_n$ in the definition of $\hat{S}^*_n$ with $Z_n^2=X_n^2I(X_n^2\le n^{2\alpha})$, 
it also follows from Lemma \ref{lemmaAA5} that
	\begin{eqnarray*}
		\sum_{k=1}^n \mathbb{E}(\hat{Z}_{k}^2) -m_2n=\left\{
		\begin{array}{ll}
			o(n^{2/(2+\delta)}), & p \in [0,2/(2+\delta)),\\
			O(n^{p}), &  p \in [2/(2+\delta),1).\\
		\end{array}\right.  
	\end{eqnarray*}
This together with (\ref{VarS})-(\ref{sasasa}) implies that
	\begin{eqnarray}\label{bnsigma}
		b_{n+1}-\sigma^2=\left\{
		\begin{array}{ll}
			o(n^{-\delta/(2+\delta)}), & p \in [0, 2/(2+\delta)),\\
			O(n^{-(1-p)}), &  p \in [2/(2+\delta),1).\\
		\end{array}\right.
	\end{eqnarray}

If $0 \le p<1/2$, then $p< 2/(2+\delta)$ and, by (\ref{Varrec}),  we have 
	\begin{eqnarray*}
		\text{Var}\big( \hat{S}^*_{n+1} \big)-\sigma^2 (n+1)/(1-2p)=\frac{n+2p}{n}\left(\mbox{Var}(\hat{S}^*_n)-\sigma^2 n/(1-2p)\right)+b_{n+1}-\sigma^2.
	\end{eqnarray*}
Hence, applying Lemma 4.2 in \cite{HZZY} and (\ref{bnsigma}) gives that 
	\begin{equation*}
		\text{Var}(\hat{S}^*_n)=\sigma^2 n/(1-2p)+
\left\{
\begin{array}{ll}
	o(n^{2/(2+\delta)}), & p \in [0, 1/(2+\delta)),\\
 o(n^{2/(2+\delta)}\log n), & p=1/(2+\delta),\\
O(n^{2p}), &  p \in (1/(2+\delta), 1/2).
\end{array}
\right.
	\end{equation*}
If $p=1/2$,  then
	\begin{eqnarray}
		&& \frac{1}{n+1}\text{Var}\big( \hat{S}^*_{n+1} \big)-\sigma^2 \log (n+1)=\frac{1}{n}\mbox{Var}(\hat{S}^*_n)-\sigma^2 \log n
+c_{n+1},  \label{recuvars}
\end{eqnarray}
where
\begin{eqnarray*}
		c_{n+1}=\frac{1}{n+1}(b_{n+1}-\sigma^2)
+\sigma^2\Big(\log\Big(1- \frac{1}{n+1}\Big)+\frac{1}{n+1}\Big).
	\end{eqnarray*}
By  (\ref{bnsigma}) and the fact that 
$\log (1-x)+x \sim (-1/2)x^2$ as $x\rightarrow 0$, 
we have $c_{n+1}=o(n^{-1-\delta/(2+\delta)})$
and hence $\sum_{n=1}^{\infty} |c_{n+1}|<\infty$.
This together with (\ref{recuvars}) implies that $\mbox{Var}(\hat{S}^*_n)=\sigma^2 n\log n+O(n)$. The proof of Lemma \ref{lemma5} is complete.
\end{proof}

	\begin{lemma} \label{lemmaa3}
If $\mathbb{E}(|X|^{r})<\infty$ with $1\le r<2$,  then
	\bestar
\frac{\hat{S}^*_{n}-\mathbb{E}(\hat{S}^*_{n})}{n^{(1+\alpha(2-r))/2}}\rightarrow 0~~\mathrm{a.s.},  \qquad \frac{\check{S}^*_{n}-\mathbb{E}(\check{S}^*_{n})}{n^{(1+\alpha(2-r))/2}}\rightarrow 0~~\mathrm{a.s.}
\eestar
\end{lemma}

\begin{proof}  Write $\beta:=(1+\alpha(2-r))/2$.
Similar arguments as in (\ref{yn4}) give that
\bestar
\mathbb{E}(\hat{Y}_{n}^2)= \hat{a}_n^2 \Big(\mathbb{E}(\hat{Z}_{n}^2)-\mathbb{E}\big( \mathbb{E}^2\big( \hat{Z}_{n} \big| \mathscr{F}_{n-1} \big)\big)\Big)\le \hat{a}_n^2 \mathbb{E}(\hat{Z}_{n}^2)\le \hat{a}_n^2 \mathbb{E}({Z}_{n}^2). 
\eestar
By  Fubini's theorem, we have
\bestar
\sum_{n=1}^{\infty}\frac{\mathbb{E}(\hat{Y}_{n}^{2})}{n^{2\beta}\hat{a}_n^2}&\le& \sum_{n=1}^{\infty}\frac{\mathbb{E}(Z_{n}^2)}{n^{2\beta}}=\sum_{n=1}^{\infty}\frac{\mathbb{E}(X_1^2I(|X_1|\le n^{\alpha}))}{n^{2\beta}}\\
&=& \mathbb{E}\Big(\sum_{n=1}^{\infty}\frac{X_1^2I(|X_1|\le n^{\alpha})}{n^{2\beta}}\Big)=\mathbb{E}\Big(X_1^2\sum_{n\ge |X_1|^{1/\alpha}}n^{-{2\beta}}\Big)\\
&\le& C\mathbb{E}(|X_1|^{2+(1-2\beta)/\alpha})=
 C\mathbb{E}(|X_1|^r)<\infty.
\eestar
This implies $\sum_{n=1}^{\infty} n^{-{2\beta}}\hat{a}_n^{-2}\mathbb{E}(\hat{Y}_n^2|\mathscr{F}_{n-1})<\infty$~a.s., and
applying Theorem 2.15 in \cite{HH1980} yields that $\sum_{n=1}^{\infty} n^{-{\beta}}\hat{a}_n^{-1}\hat{Y}_n$ converges a.s. By Kronecker's lemma, we have
\bestar
\frac{\hat{S}^*_{n}-\mathbb{E}(\hat{S}^*_{n})}{n^{\beta}}=\frac{\sum_{k=1}^n \hat{Y}_k}{n^{\beta}\hat{a}_n}\rightarrow 0~~\mathrm{a.s.}
\eestar
Similarly, we can obtain that
\bestar
\frac{\check{S}^*_{n}-\mathbb{E}(\check{S}^*_{n})}{n^{\beta}}\rightarrow 0~~\mathrm{a.s.}
\eestar
The proof of Lemma \ref{lemmaa3} is complete.
\end{proof}

	\begin{lemma} \label{lemmaa4}
Suppose that $0\le \delta<2$ and $\alpha=1/(2+\delta)$.
If $\mathbb{E}(|X|^{2+\delta})<\infty$,  then
	\bestar
&&\mathbb{E}(\hat{Y}_n^2|\mathscr{F}_{n-1})
-\mathbb{E}(\hat{Y}_n^2)=o(1)\hat{a}_n^2n^{-\delta/(2+\delta)}~~~~\mathrm{a.s.}\\
&&\mathbb{E}(\check{Y}_n^2|\mathscr{F}_{n-1})
-\mathbb{E}(\check{Y}_n^2)=o(1)\check{a}_n^2n^{-\delta/(2+\delta)}~~~~\mathrm{a.s.}
\eestar
\end{lemma}

\begin{proof} Observe that
\be
\mathbb{E}(\hat{Y}_n^2|\mathscr{F}_{n-1})
=\hat{a}_n^2 ( \mathbb{E}( \hat{Z}_{n}^2 | \mathscr{F}_{n-1} ) - \mathbb{E}^2( \hat{Z}_{n} | \mathscr{F}_{n-1} )),  \label{EYN2}
\ee
and
\bestar
\mathbb{E}(\hat{Z}_{n+1}^i|\mathscr{F}_{n})=(1-p) \mathbb{E}(Z_{n+1}^i)+\frac{p}{n}\sum_{k=1}^n \hat{Z}_{k}^i,~~~~~~i=1,2.
\eestar
Hence
\bestar
\mathbb{E}(\hat{Z}_{n+1}^i|\mathscr{F}_{n})-\mathbb{E}(\hat{Z}_{n+1}^i)=\frac{p}{n}\sum_{k=1}^n (\hat{Z}_{k}^i
-\mathbb{E}(\hat{Z}_{k}^i)),~~~~~~i=1,2.
\eestar
By Lemma \ref{lemmaa3},   we have that for any $\varepsilon>0$,
\bestar
\hat{S}^*_{n}-\mathbb{E}(\hat{S}^*_{n})=o(n^{1/2+\varepsilon})~~~~\mathrm{a.s.}
\eestar
This together with  (\ref{Varsn}), and  Lemmas \ref{lemmaAA5} and \ref{lemmaa3} implies that for any $\varepsilon>0$,
	\begin{eqnarray}
	&&\mathbb{E}^2(\hat{Z}_{n+1}|\mathscr{F}_{n})-\mathbb{ E}(\mathbb{E}^2(\hat{Z}_{n+1}|\mathscr{F}_{n}))\nonumber\\
	&&~~~~= \left(\frac{p}{n}\hat{S}^*_n + (1-p)\mathbb{E}(Z_{n+1})\right)^2-\mathbb{ E}\left(\frac{p}{n}\hat{S}^*_n + (1-p)\mathbb{E}(Z_{n+1})\right)^2\nonumber\\
	&&~~~~=2p(1-p)n^{-1}\mathbb{ E}(Z_{n+1})(\hat{S}^*_n-\mathbb{ E}\hat{S}^*_n)+p^2n^{-2}\left((\hat{S}_n^*)^2-\mathbb{ E}(\hat{S}_n^*)^2\right)\nonumber\\
	&&~~~~=o(n^{-1/2+\varepsilon})+p^2n^{-2}\left((\hat{S}_n^*-\mathbb{ E}\hat{S}_n^*)(\hat{S}_n^*+\mathbb{ E}\hat{S}_n^*)-\text{Var}(S_n^*)\right)\nonumber\\
	&&~~~~=o(n^{-1/2+\varepsilon})~~a.s. \label{key111}
\end{eqnarray}

By replacing $Z_n$ in the definition of $\hat{S}^*_n$ with $Z_n^2=X_n^2I(X_n^2\le n^{2\alpha})$, 
it also follows from Lemma \ref{lemmaa3}  that
\bestar
\mathbb{E}(\hat{Z}_{n+1}^2|\mathscr{F}_{n})
-\mathbb{E}(\hat{Z}_{n+1}^2)=\frac{p}{n}\sum_{k=1}^n (\hat{Z}_{k}^2
-\mathbb{E}(\hat{Z}_{k}^2))
=o(n^{(\alpha(2-\delta)-1)/2})=o(n^{-\delta/(2+\delta)})~~~~\mathrm{a.s.}
\eestar
Combining this with (\ref{EYN2}) and (\ref{key111}) yields that
\bestar
&&\mathbb{E}(\hat{Y}_n^2|\mathscr{F}_{n-1})
-\mathbb{E}(\hat{Y}_n^2)\\
&&~~~~=\hat{a}_n^2 ( \mathbb{E}( \hat{Z}_{n}^2 | \mathscr{F}_{n-1})-\mathbb{E}(\hat{Z}_{n}^2 )  - (\mathbb{E}^2(\hat{Z}_{n}|\mathscr{F}_{n-1})-\mathbb{ E}(\mathbb{E}^2(\hat{Z}_{n}|\mathscr{F}_{n-1}))))\\
&&~~~~=o(1)\hat{a}_n^2n^{-\delta/(2+\delta)}~~~~\mathrm{a.s.}
\eestar
Similarly,  we have
\begin{equation*}
	\mathbb{E}(\check{Y}_n^2|\mathscr{F}_{n-1})
	-\mathbb{E}(\check{Y}_n^2)=o(1)\check{a}_n^2n^{-\delta/(2+\delta)}~~~\mathrm{a.s.}
\end{equation*}  The proof of Lemma \ref{lemmaa4} is complete.
\end{proof}

In order to deal with the remaining after truncation,
we use the idea in \cite{HZZY} and represent $\hat{S}_n, \hat{S}^*_n, \check{S}_n, \check{S}^*_n$ as
\be
&& \hat{S}_n=\sum_{j=1}^{n} N_j(n)X_j,~~~~\hat{S}^*_n=\sum_{j=1}^{n} N_j(n)Z_j, \label{hatsexpress}\\
&&\check{S}_n=\sum_{j=1}^{n} \Delta_j(n)X_j,~~~~\check{S}^*_n=\sum_{j=1}^{n} \Delta_j(n)Z_j,\nonumber
\ee
where $	N_j(n)$ and $\Delta_j(n)$ are defined as in Section 4 of
\cite{HZZY}. In order to prove Propositions \ref{mainthm212}
and \ref{mainthm222}, we also need to extend Lemmas  4.3 and 4.4 in \cite{HZZY}, which provide
 the strong limit results 
on $N_j(n)$ and $\Delta_j(n)$.

\begin{lemma} \label{Njn} For any   $p\in [0, 1)$, we have
	\begin{equation*}\label{diffe11}
		\frac{N_j(n)}{n^{p}\log \log  n} \rightarrow 0 \quad\mathrm{a.s.},
	\end{equation*} 
and
\begin{equation*}\label{diffe12}
	\frac{\Delta_j(n)}{n^{p/2}\log  n} \rightarrow 0 \quad\mathrm{a.s.}
\end{equation*} 
\end{lemma}

\begin{proof}
	From the proof of Lemma 4.3 in \cite{HZZY},
we have 
	\bestar
	\mathbb{E}\big( N_j(n+1) \big| \mathscr{F}_n \big)=\frac{n+p}{n} N_j(n), \label{enj}
	\eestar
and  $\mathbb{E}(N_j(n))=O(n^p)$.
Let  $\omega_n= n^{p}\log \log  n$.	 Then $	\mathbb{E}(N_j(n))/\omega_n\rightarrow 0$, and hence  $N_j(n)/\omega_n\stackrel{L_1}{\rightarrow} 0$  since $N_j(n)\ge 0 $.

Next, we will show that $\frac{(n+p)\omega_n}{n\omega_{n+1}}<1$ for sufficiently large $n$.
Note that
\bestar
\log \frac{(n+p)\omega_n}{n\omega_{n+1}}
=\log (1+p/n)-p\log (1+1/n)+ \log \log \log  n-\log\log \log (n+1).
\eestar	
By L'Hospital's rule, we have
\bestar
\lim_{x\rightarrow 0}\frac{\log(1+px)-p\log (1+x)}{x^2}
&=&\lim_{x\rightarrow 0}\frac{p/(1+px)-p/(1+x)}{2x}\\
&=&\lim_{x\rightarrow 0}\frac{p(1-p)}{2(1+px)(1+x)} = p(1-p)/2.
\eestar
And  $\log\log \log (n+1)-\log \log\log n=1/(\xi_n \log \xi_n \log\log \xi_n)\sim 1/(n\log n \log\log n)$ by the mean value theorem, where
$\xi_n\in (n, n+1)$. Hence
$(n\log n \log\log n)\log \frac{(n+p)\omega_n}{n\omega_{n+1}}\rightarrow -1$ and the desired result holds. 

Then, for sufficiently large $n$, we have
	\bestar
	\mathbb{E}\Big( \frac{N_j(n+1)}{\omega_{n+1}} \Big| \mathscr{F}_n \Big) = \frac{(n+p)\omega_n}{n\omega_{n+1}}\frac{N_j(n)}{\omega_n}\le \frac{N_j(n)}{\omega_n}.
	\eestar
	Hence there exists $n_0\in \mathbb{N}$ such that $\{N_j(n)/\omega_n, n\ge n_0\} $ is a nonnegative supermartingale,
	and  the $L_1$
convergence implies the almost sure convergence.

By using the same arguments as in  the proof of Lemma 4.4 in \cite{HZZY} but more precise calculations, we 
can obtain that $\Delta_j(n)/(n^{p/2}\log n)\rightarrow 0$~a.s.
\end{proof}

Now we turn to prove Propositions \ref{mainthm212}
and \ref{mainthm222}.

\begin{proof}[ Proof of Proposition \ref{mainthm212}]
Choose $\alpha=1/(2+\delta)$ and define
\bestar
\hat{c}_n^2
=\left\{
\begin{array}{ll}
\hat{a}_n^2n^{2/(2+\delta)}, & p\in [0, 1/(2+\delta)),\\
 \log n, & p=1/(2+\delta),\\
\log\log n, &p \in (1/(2+\delta),  1/2].
\end{array}
\right.
\eestar
By   (\ref{asy of hat a_n}),
we have
$\sum_{k=1}^n \hat{a}_k^2k^{-\delta/(2+\delta)}
=O(\hat{c}_n^2)$, and hence by  Lemma  \ref{lemmaa4},
\bestar
\sum_{k=1}^n (\mathbb{E}(\hat{Y}_k^2|\mathscr{F}_{k-1})-\mathbb{E}(\hat{Y}_k^2))=o(\hat{c}_n^2)
~~\mathrm{a.s.}
\eestar
Applying (\ref{yn4}) and Fubini's theorem gives that
\begin{eqnarray*}
	\sum_{k=1}^\infty \hat{c}_k^{-4}\mathbb{E}( \hat{Y}_k^{4}) &\le& C \sum_{k=1}^\infty\frac{\mathbb{E}( \hat{Y}_k^{4})}{\hat{a}_k^4 k^{4/(2+\delta)}}
	\le  	C\sum_{k=1}^\infty  k^{-4/(2+\delta)}\mathbb{E}(X_1^{4}I(|X_1|\le k^{1/(2+\delta)}))\nonumber\\
	&\le&C\mathbb{E}\left(X_1^{4}\sum_{k=1}^\infty{k^{-4/(2+\delta)}}I(|X_1|\le k^{1/(2+\delta)})\right)\nonumber\\
	&=&C\mathbb{ E}\left(X_1^{4}\sum_{k \ge |X_1|^{2+\delta}} k^{-4/(2+\delta)}\right)\nonumber\\
	&\le&C\mathbb{ E}|X_1|^{4+(2+\delta)(1-4/(2+\delta))}= C\mathbb{ E}|X_1|^{2+\delta}<\infty,
\end{eqnarray*}
where we have used the fact that $\hat{a}_n^4 n^{4/(2+\delta)}=O(\hat{c}_n^4).$
Then by Theorem 2.1 of Shao \cite{Shao1993},  we can redefine
$\{\hat{M}^*_n, n\ge 1\}$ on  a richer probability space
on which there exists a standard Brownian motion $\{B(t), t\ge 0\}$ such that
\bestar
\hat{M}^*_n-B(\hat{\sigma}_{1n}^2)=o\Big(\hat{c}_n\Big(\log \frac{\hat{\sigma}_{1n}^2}{\hat{c}_n^2}+\log\log \hat{c}_n^2\Big)^{1/2}\Big)=\left
\{\begin{array}{ll}
o(\hat{c}_n\sqrt{\log n}), & p \in [0, 1/2),\\
o(\hat{c}_n\sqrt{\log \log n}), & p=1/2,
\end{array}
\right.
~~\mathrm{a.s.},
\eestar
where $\hat{\sigma}_{1n}^2=\mbox{Var}(\hat{M}^*_n)=\hat{a}_n^2\mbox{Var}(\hat{S}^*_n)$, and we have used (\ref{asy of hat a_n}) and Lemma  \ref{lemma5}.
Furthermore,
\bestar
\frac{\hat{S}^*_n-\mathbb{E}(\hat{S}^*_n)}{\sqrt{\mbox{Var}(\hat{S}^*_n)}}-\frac{B(\hat{\sigma}_{1n}^2)}{\hat{\sigma}_{1n}}=\frac{1}{\hat{\sigma}_{1n}}(\hat{M}^*_n-B(\hat{\sigma}_{1n}^2))=o(\hat{\delta}_n)
~~\mathrm{a.s.},
\eestar
where $\hat{\delta}_n$ is defined as in (\ref{Nas}).
This together with Lemmas  \ref{lemmaAA5} and \ref{lemma5}, and the law of iterated 
logarithm for a Brownian motion implies that
\begin{eqnarray}
\Big|\frac{\hat{S}^*_n-m_1n}{\hat{\sigma}_n}-\frac{B(\hat{\sigma}_{1n}^2)}{\hat{\sigma}_{1n}}\Big|
&\le& \frac{\sqrt{\mbox{Var}(\hat{S}^*_n)}}{\hat{\sigma}_n}\Big|\frac{\hat{S}^*_n-\mathbb{E}(\hat{S}^*_n)}{\sqrt{\mbox{Var}(\hat{S}^*_n)}}-\frac{B(\hat{\sigma}_{1n}^2)}{\hat{\sigma}_{1n}}\Big| 
\nonumber\\
&&+\frac{|\mathbb{E}(\hat{S}^*_n)-m_1n|}{\hat{\sigma}_{n}}+\Big|\frac{\sqrt{\mbox{Var}(\hat{S}^*_n)}}{\hat{\sigma}_n}-1\Big|\frac{|B(\hat{\sigma}_{1n}^2)|}{\hat{\sigma}_{1n}}\nonumber\\
&=&o(\hat{\delta}_{n})
~~\mathrm{a.s.},   \label{diffsBM}
\end{eqnarray}
where  $\hat{\sigma}_n^2$ is defined as in (\ref{evarhat2}).

By Fubini's theorem, we have
	\bestar
	\sum_{n=1}^{\infty}\mathbb{P}(Z_n\ne X_n)&=&\sum_{n=1}^{\infty}\mathbb{P}(|X_n|>n^{1/(2+\delta)})=\sum_{n=1}^{\infty}\mathbb{E}(I(|X_1|^{2+\delta}>n))\\
	&=&\mathbb{E}\Big(\sum_{n=1}^{\infty}I(|X_1|^{2+\delta}>n)\Big)\le \mathbb{E}(|X_1|^{2+\delta})<\infty,
	\eestar
	and hence $\mathbb{P}(Z_n\ne X_n~\mathrm{i.o.})=0$ by the Borel-Cantelli Lemma. By
applying Lemma 4.5 in \cite{HZZY} and using (\ref{evarhat2}), (\ref{Nas}), (\ref{hatsexpress}) and Lemma \ref{Njn}, we have
	\be
	\frac{1}{\hat{\sigma}_n \hat{\delta}_n} (\hat{S}^*_n-\hat{S}_n)= \frac{n^{p}\log\log n}{\hat{\sigma}_n \hat{\delta}_n}\sum_{j=1}^n \frac{N_j(n)}{n^{p}\log\log n} (Z_j-X_j)\rightarrow 0\quad\mathrm{a.s.} 
\label{diffstars}
	\ee

Note that
\bestar
\hat{\sigma}_{1n}^2=\hat{a}_n^2\mbox{Var}(\hat{S}^*_n)
=\hat{\sigma}_{2n}^2 (1+\hat{\delta}_{1n}),
\eestar
where $\hat{\sigma}_{2n}^2=\Gamma^2(1+p)n^{-2p}\hat{\sigma}_n^2$ and
\bestar
\hat{\delta}_{1n}=\left\{
\begin{array}{ll}
o(n^{-\delta/(2+\delta)}), & p \in [0, 1/(2+\delta)),\\
o(n^{-\delta/(2+\delta)}\log n), & p=1/(2+\delta),\\
O(n^{2p-1}), &p\in (1/(2+\delta), 1/2),\\
O(1/\log n), & p=1/2.
\end{array}
\right.
\eestar
Then by Theorem 1.2.1 in \cite{CR1981}, we have
\bestar
\Big|\frac{B(\hat{\sigma}_{1n}^2)}{\hat{\sigma}_{1n}}-\frac{B(\hat{\sigma}_{2n}^2)}{\hat{\sigma}_{2n}}\Big|
&\le&  \frac{|B(\hat{\sigma}_{1n}^2)-B(\hat{\sigma}_{2n}^2)|}{\hat{\sigma}_{2n}}+|B(\hat{\sigma}_{2n}^2)||\hat{\sigma}_{1n}^{-1}-\hat{\sigma}_{2n}^{-1}|\\
&=&  \frac{|B(\hat{\sigma}_{1n}^2)-B(\hat{\sigma}_{2n}^2)|}{\hat{\sigma}_{2n}}+\frac{|B(\hat{\sigma}_{2n}^2)||\hat{\sigma}_{1n}^2-\hat{\sigma}_{2n}^2|}{\hat{\sigma}_{1n}\hat{\sigma}_{2n}(\hat{\sigma}_{1n}+\hat{\sigma}_{2n})}\\
&=& O(1)\hat{\delta}_{1n}(-\log \hat{\delta}_{1n}+\log\log \hat{\sigma}_{2n}^2)^{1/2}=o(\hat{\delta}_n)~~\mathrm{a.s.}
\eestar
This together with (\ref{diffsBM}) and (\ref{diffstars}) implies that
\bestar
\Big|\frac{\hat{S}_n-m_1n}{\hat{\sigma}_n}-\frac{B(\hat{\sigma}_{2n}^2)}{\hat{\sigma}_{2n}}\Big|=o(\hat{\delta}_n)~~\mathrm{a.s.}
\eestar
Then (\ref{spt}) follows by defining
$W(t)=B(\Gamma^2(1+p)t)/\Gamma(1+p), t\ge 0,$
which is a standard Brownian motion.
\end{proof}

\begin{proof}[ Proof of Proposition \ref{mainthm222}]
The proof is similar to that of  Proposition \ref{mainthm212} and we omit the details.
\end{proof}

\end{appendices}

\end{document}